\documentclass[12pt]{amsart}
\usepackage[top=4.2cm, bottom=4cm, left=2.4cm, right=2.4cm]{geometry}
\usepackage[utf8]{inputenc}
\usepackage[USenglish]{babel}
\usepackage[T1]{fontenc} 
\usepackage{mathrsfs}
\usepackage{mathtools}
\usepackage{amsmath}
\usepackage{amssymb,epsfig}
\usepackage{caption}
\usepackage{subcaption}
\usepackage{quiver}

\usepackage{amsthm}
\usepackage[absolute]{textpos}
\usepackage{mathtools,emptypage}

\usepackage[bookmarks=true]{hyperref}
\usepackage{xcolor}
\hypersetup{
    colorlinks,
    linkcolor={red!50!black},
    citecolor={blue!50!black},
    urlcolor={blue!80!black},
}

\usepackage{todonotes}

\mathtoolsset{showonlyrefs} % uncomment to tag only the referred equations (post-production) 

\newcommand{\N}{\mathbb{N}}
\newcommand{\Q}{\mathbb{Q}}
\newcommand{\R}{\mathbb{R}}

\newcommand{\sfd}{{\sf d}}

\renewcommand{\d}{{\mathrm d}}

\newcommand{\restr}[1]{\lower3pt\hbox{$|_{#1}$}}

\newcommand{\limi}{\varliminf}
\newcommand{\lims}{\varlimsup}
                          
\newcommand{\X}{{\rm X}}

\newcommand{\lip}{{\rm lip}}
\newcommand{\Lip}{{\rm Lip}}
\newcommand{\loc}{{\rm loc}}

\newcommand{\mm}{\mathfrak m}

\renewcommand{\SS}{{\rm SS}}

\makeatletter
\newcommand{\mytag}[2]{%
  \text{#1}%
  \@bsphack
  \begingroup
    \@onelevel@sanitize\@currentlabelname
    \edef\@currentlabelname{%
      \expandafter\strip@period\@currentlabelname\relax.\relax\@@@%
    }%
    \protected@write\@auxout{}{%
      \string\newlabel{#2}{%
        {\color{black}#1}%
        {\thepage}%
        {\@currentlabelname}%
        {\@currentHref}{}%
      }%
    }%
  \endgroup
  \@esphack
}
\makeatother

\def\Xint#1{\mathchoice
{\XXint\displaystyle\textstyle{#1}}%
{\XXint\textstyle\scriptstyle{#1}}%
{\XXint\scriptstyle\scriptscriptstyle{#1}}%
{\XXint\scriptscriptstyle\scriptscriptstyle{#1}}%
\!\int}
\def\XXint#1#2#3{{\setbox0=\hbox{$#1{#2#3}{\int}$ }
\vcenter{\hbox{$#2#3$ }}\kern-.6\wd0}}

\def\dashint{\Xint-}

\allowdisplaybreaks

\newtheorem{theorem}{Theorem}[section]

\newtheorem{corollary}[theorem]{Corollary}
\newtheorem{lemma}[theorem]{Lemma}
\newtheorem{proposition}[theorem]{Proposition}
\theoremstyle{definition}
\newtheorem{definition}[theorem]{Definition}

\newtheorem{example}[theorem]{Example}

\newcounter{Counter}

\newtheorem{remark}[theorem]{Remark}

\newcommand{\Int}{\textup{int}}

\newcommand{\codH}[1]{\mathcal{H}^{{\rm cod-}{#1}}}

\title{Poincar\'{e} inequality and energy of separating sets}

\author[Emanuele Caputo]{Emanuele Caputo}\address[Emanuele Caputo]{Mathematics Institute, Zeeman Building, University of Warwick, Coventry, CV4 7AL, United Kingdom}\email{emanuele.e.caputo@jyu.fi}
\author[Nicola Cavallucci]{Nicola Cavallucci}\address[Nicola Cavallucci]{
Institute of Mathematics, EPFL, Station 8, 1015 Lausanne, Switzerland}\email{n.cavallucci23@gmail.com}

\keywords{Poincar\'{e} inequality, Separating sets, Perimeter, Minkowski content, metric measure spaces}
\subjclass[2020]{30L15, 53C23, 49J52}

\begin{document}

\maketitle

\begin{abstract}
    We study geometric characterizations of the Poincar\'{e} inequality in doubling metric measure spaces in terms of properties of separating sets.
    Given a couple of points and a set separating them, such properties are formulated in terms of several possible notions of energy of the boundary, involving for instance the perimeter, codimension type Hausdorff measures, capacity, Minkowski content and approximate modulus of suitable families of curves. We prove the equivalence within each of these conditions and the $1$-Poincaré inequality.
\end{abstract}

\tableofcontents

\section{Introduction}
Since the work of Heinonen-Koskela \cite{HeiKos98} it has been clear the importance of the class of doubling metric measure spaces satisfying a Poincar\'{e} inequality, called for short PI-spaces. On such spaces many analytical properties that were classically studied for Euclidean spaces or Riemannian manifolds hold (see for instance the monographs \cite{BjornBjorn2011} and \cite{HK00}). However, despite its definition being analytical, it was clear from the beginning of the theory that a Poincaré inequality has strong geometric consequences. One of the most notable result in such a direction is Cheeger's theorem (\cite{Cheeger99}) saying that PI spaces are Lipschitz differentiability spaces. Moreover, as noticed by Semmes, every such space must be $L$-quasiconvex, namely there exists a constant $L \ge 1$, such that each couple of points can be connected by a $L$-quasigeodesic, i.e.\ a curve whose length is bounded by $L$ times the distance of the points (see for instance \cite[Thm.\ 17.1]{Cheeger99}). The relation between the validity of a Poincaré inequality and the presence of `many' curves in the space was highlighted in \cite{Kei03} by Keith. He proved that, given a doubling metric measure space $(\X,\sfd,\mm)$, a $p$-Poincaré inequality is equivalent to a uniform lower bound on the $p$-modulus of curves joining two points $x,y \in \X$, but computed with respect to the weighted measure $\mm_{x,y}^L = R_{x,y}^L \mm$, where
\begin{equation*}
    R_{x,y}^L(z):= \chi_{B_{x,y}^L}(z) \left( \frac{\sfd(x,z)}{\mm(B_{\sfd(x,z)}(x))} + \frac{\sfd(y,z)}{\mm(B_{\sfd(y,z)}(y))}\right).
\end{equation*}
Here $B_{x,y}^L := B_{2L\sfd(x,y)}(x)\cup B_{2L\sfd(x,y)}(y)$. We call $R_{x,y}^L$ the \emph{$L$-truncated Riesz potential with poles at $x,y$}.
The Riesz potential, together with the Hardy-Littlewood maximal function, has been a central tool in the theory of Poincaré inequalities from the beginning (\cite{Hei01}, \cite{HK00}).\\
Semmes conjectured that the $1$-Poincar\'{e} inequality is equivalent to having the existence of a $1$-pencil of curves for every couple of points, that is a measure on quasigeodesics connecting them with bounded overlap controlled in terms of the Riesz potential (see Proposition \ref{prop:Poincaré_equivalences_pencil} for the precise definition).
The conjecture is true as proven independently by \cite{FasslerOrponen19} and \cite{DurCarErikBiqueKorteShanmu21}. The proof of the latter one can be extended for $p>1$ as we notice in \cite{CaputoCavallucci2024}.  There we also clarify the roles of the different constants in the definitions.\\
Another remarkable condition that encodes the Poincaré inequality for $p>1$ in geometric terms is defined through an obstacle avoiding condition. The condition is informally saying that for each couple of points and for each Borel set that is not too dense around these two points (in a sense quantified with the maximal function) there exists a quasigeodesic connecting the two points that does not spend too much time inside the set, see \cite{ErikssonBique2019} and \cite{Sylvester-Gong-21}. The same condition for $p=1$ is studied in detail in \cite{CaputoCavallucci2024}, where we clarify its relation with the $1$-Poincaré inequality.
\vspace{2mm}

All the conditions above turned out to be useful in order to prove that relevant class of spaces appearing in metric measure geometry and analysis on fractals satisfy or do not satisfy a Poincar\'{e} inequality. All the proofs of these facts use a characterization which is relevant for the geometry under study.
For instance, using the obstacle-avoidance approach \cite{Sylvester-Gong-21} studied the Poincar\'{e} inequality on non-self similar Sierpinski carpets and more general fractals (see \cite{MacKayTysonWildrick} for previous results in this direction). 
Let us comment with an application of the $1$-pencil of curves. Ohta \cite{Ohta07} and Sturm \cite{SturmII} introduced a notion of metric measure spaces satisfying a weak lower bound on the Ricci curvature by $K$ and an upper bound on the dimension by $N$ in a synthetic sense, the so called ${\rm MCP}(K,N)$ spaces. A straightforward consequence of the definition of ${\rm MCP}(0,N)$ spaces implies the existence of $1$-pencils of curves and so a $1$-Poincaré inequality, under the additional assumption of $N$-Ahlfors regularity.\\

The goal of our paper is to go in this direction and provide a new characterizations of the Poincar\'{e} inequality in geometric terms.
Let us notice that all the previous characterizations are defined in terms of properties of curves. The novelty of our approach is that our equivalent conditions are not formulated using curves but using properties related to separating sets. Let us introduce such a notion.
\begin{definition}
\label{def:separating_sets}
    Let $(\X,\sfd,\mm)$ be a metric measure space and let $x,y\in \X$. A closed set $\Omega$ is a separating set from $x$ to $y$ if there exists $r>0$ such that $B_r(x) \subseteq \Int(\Omega)$ and $B_r(y) \subseteq \Omega^c$. We denote by $\SS_{\textup{top}}(x,y)$ the class of all separating sets from $x$ to $y$.
\end{definition}
Here the subscript `${\rm top}$' stands for topological, in contrast with other notions of separating sets studied in \cite{CaputoCavallucci2024}.
We state the main result of the paper. It shows that for $p=1$ the Poincaré inequality is equivalent to have a uniform lower bound on several energies of boundaries of separating sets. These energies are all defined through the Riesz potential.
\begin{theorem}
\label{theo:main-intro-p=1}
Let $(\X,\sfd,\mm)$ be a doubling metric measure space. Then the following conditions are quantitatively equivalent:
\begin{itemize}
    \item[(PI)] $1$-Poincar\'{e} inequality;
    \item[(BP)] $\exists c>0, L\geq 1$ such that $\int R_{x,y}^L\,\d {\rm Per}_\mm(\Omega, \cdot) \ge c$ for every $x, y \in \X$ and $\Omega \in \SS_{\textup{top}}(x,y)$;
    \item[(BC)] $\exists c>0, L\geq 1$ such that $\textup{cap}_{\mm_{x,y}^L}(\Omega) \ge c$ for every $x, y \in \X$ and   $\Omega \in \SS_{\textup{top}}(x,y)$;
    %\item[(e-BC)] $\exists c>0, L\geq 1$ such that $\textup{cap}_{\mm_{x,y}^L}(\partial^e \Omega) \ge c$ for every $x, y \in \X$ and for every $\Omega \in \SS_{\textup{top}}(x,y)$;
    \item[(BH)] $\exists c>0, L\geq 1$ such that $\int_{\partial \Omega} R_{x,y}^L\,\d \codH{1}_\mm \ge c$ for every $x, y \in \X$ and  $\Omega \in \SS_{\textup{top}}(x,y)$;
    \item[(BMC)] $\exists c>0, L\geq 1$ such that $(\mm_{x,y}^L)^{+,1}(\Omega) \ge c$ for every $x, y \in \X$ and  $\Omega \in \SS_{\textup{top}}(x,y)$;
    %\item[(e-BMC)] $\exists c>0, L\geq 1$ such that $(\mm_{x,y}^L)^{+,1}(\partial^e \Omega) \ge c$ for every $x, y \in \X$ and for every $\Omega \in \SS_{\textup{top}}(x,y)$;
    \item[(BAM)] $\exists c>0, L\geq 1$ such that ${\rm AM}_1(\partial \Omega, \mm_{x,y}^L) \ge c$ for every $x, y \in \X$ and  $\Omega \in \SS_{\textup{top}}(x,y)$;
    \item[(BAM$^\pitchfork$)] $\exists c>0, L\geq 1$ such that ${\rm AM}_1^\pitchfork(\Omega, \mm_{x,y}^L) \ge c$ for every $x, y \in \X$ and $\Omega \in \SS_{\textup{top}}(x,y)$.
\end{itemize}
\end{theorem}
We refer the reader to Section \ref{sec:several_energies} for the definition of all the quantities involved in the statement. The expert reader can recognize the perimeter, the variational capacity, the Minkowski content, the codimension-$1$ Hausdorff measure and the $1$-approximate modulus. 
The acronyms of each line should be read accordingly: the letter B stands for Big, therefore for instance (BP) means `Big Perimeter', while (BH) means `Big codimension-$1$ Hausdorff measure'. Similarly for the other quantities.
We refer the reader to Theorem \ref{theo:main-intro-p=1-riproposed} where additional equivalent conditions are presented. In particular we stress the fact that all the conditions above that involve the topological boundary of a separating set can be checked on its essential boundary. Moreover we defined other equivalent conditions by defining the perimeter and the codimension-$1$ Hausdorff measure directly on the metric measure space $(\X,\sfd,\mm_{x,y}^L)$.
\\

Let us give an intuition about what these conditions in the main statement mean. Let us start with (BH) and analyze the special case of $(\X,\sfd,\mm)$ being $s$-Ahlfors regular, with $s \ge 1$. In this case, the condition (BH) implies that for every $x,y \in \X$ and every $\Omega \in {\rm SS}_{\rm top}(x,y)$, $\mathcal{H}^{s-1}(\partial \Omega) >0$. Thus in particular the Hausdorff dimension of $\partial \Omega$ is at least $s-1$. This says that the Poincar\'{e} inequality implies that the boundary of separating sets are big in terms of Hausdorff dimension.
Actually, the (BH) says more than what just stated and quantifies how big the $(s-1)$ Hausdorff measure of $\partial \Omega$ is, taking into account the distance from the poles.
This can be better understood if one thinks about separating sets in the Euclidean case $\mathbb{R}^d$ endowed with Lebesgue measure, which can be regarded as the simplest example of $1$-PI space. Here, considering two points $x,y \in \mathbb{R}^d$, sufficiently small balls around $x$ are all separating sets from $x$ to $y$ and, while the $\mathcal{H}^{d-1}$ measure of the spheres converges to $0$, the weighted Hausdorff measure with the Riesz potential $R_{x,y}^L\,\mathcal{H}^{d-1}$ is uniformly bounded from below.\\

Let us discuss the condition (BMC) and its relation with the main proof in \cite{FasslerOrponen19}. In order to prove that $1$-PI spaces supports a $1$-pencil of curves for each couple of points $x,y\in \X$, the authors argue as follows: they discretize the metric measure space by considering a $\delta$-net for $\delta$ small and they build a graph connecting close points in the $\delta$-net. Then considering all possible cuts in the graph, namely a subset of vertices containing $x$ and not $y$, they prove that a suitable energy of the cut weighted with the Riesz potential is uniformly bounded from below. The bound from below is independent of $x,y,\delta$ and can be regarded as a discretized version of our characterizations, for instance the closest one being (BMC). One of the goals of \cite{CaputoCavallucci2024} is to investigate the relation between boundary of cuts in the appoximating graphs and boundary of separating sets in the continuous version.\\

The conditions (BAM) and (BAM$^\pitchfork$), since they are formulated in terms of modulus of curves, share some similarities to Keith's characterization of $p$-PI space in \cite{Kei03}. First of all we prove along the way, see Remark \ref{rem:keith_similarity}, that $\X$ is a $1$-PI space if and only if for every $x,y \in \X$ we have ${\rm AM}_1(\Gamma_{x,y}^L,\mm_{x,y}^L) \ge c$ for some $c>0$ and $L \ge 1$. This is a strengthening of Keith's characterization for $p =1$. On the other hand (BAM) and (BAM$^\pitchfork)$ involves the class of curves hitting the boundary of a separating set, which is a much larger family of curves. However the validity of this condition for every separating set, which can be equivalently stated as a property of arbitrarily short curves (see Lemma \ref{lemma:AM_localized_near_boundary}), is enough to get a uniform lower bound on ${\rm AM}_1(\Gamma_{x,y}^L,\mm_{x,y}^L)$.\\ 

There are two natural research directions starting from our result. The first one is the case $p>1$. We prove in Proposition \ref{prop:PI_implies_pBH} that the natural $p$-generalization of (BH) is satisfied by a doubling metric measure space with a $p$-Poincar\'{e} inequality. Example \ref{example:sierpinski_p>1} shows that this version of (BH) does not imply any Poincar\'{e} inequality. It would be interesting to find additional assumptions on the metric measure space so that the condition appearing in Proposition \ref{prop:PI_implies_pBH} characterizes $p$-PI spaces.
The second direction is whether it is possible to test each of the properties appearing in Theorem \ref{theo:main-intro-p=1} on a subclass of more regular separating sets. The regularity is such sets can be expressed as sets with Lipschitz (or even smooth) boundaries in the case of Riemannian manifolds or in terms of metric measure regularity conditions, such as the interior (or exterior) ball condition and similar properties (see the discussion in \cite[§4.2]{CaputoRossi}). 

\subsection{Structure of the paper}
The paper is structured as follows. In Section \ref{sec:preliminaries} we recall definitions and properties we will need in the presentation. We also study the measure $\mm_{x,y}^L$ and its known relations with the Poincaré inequality. In Section \ref{sec:several_energies} we introduce the several energies appearing in Theorem \ref{theo:main-intro-p=1} and we prove some basic relations between them holding in full generality. In Sections \ref{sec:relations_under_PI} and \ref{sec:relations_weighted_with_riesz} we specialize the study of the relations between our different energies in case of spaces satisfying a Poincaré inequality. These sections are quite technical, especially Section \ref{sec:relations_weighted_with_riesz} in which we compare two possible ways to weight the perimeter measure and the codimension Hausdorff measure with the Riesz potential. Finally in Section \ref{sec:main_theorem} we prove all the equivalences stated in Theorem \ref{theo:main-intro-p=1}. The core of the arguments are contained in the proof of the main theorem. In particular the reader can skip the proofs of Section \ref{sec:relations_weighted_with_riesz} during a first read.

\subsection{Acknowledgments}

The first author was supported by Academy of Finland, grants no. 321896. He currently acknowledges the support by the European Union’s Horizon 2020 research and innovation programme (Grant
agreement No. 948021).
He also thanks the kind hospitality of the Univerisity of Ottawa and of Augusto Gerolin. We thank Francesco Nobili for a careful reading of the manuscript.

\section{Preliminaries}
\label{sec:preliminaries}
Let $(\X,\sfd,\mm)$ be a metric measure space, i.e.\ $(\X,\sfd)$ is a complete and separable metric space and $\mm$ is nonnegative Borel measure that is finite on bounded sets. Given $A\subseteq \X$ and $r\geq 0$ we denote by $B_r(A):=\{ x:\, \sfd(x,A) < r\}$ and by $\overline{B}_r(A):=\{ x:\, \sfd(x,A) \le r\}$, i.e.\ respectively the open and closed $r$-neighbourhood of $A$. For instance the open ball of center $x \in \X$ and radius $r$ is $B_r(x)$. We denote by $\mathscr{B}(\X)$ the class of Borel subsets of $\X$.\\
We denote by ${\rm Lip}(\X)$ the space of Lipschitz functions on $\X$ with values in $\R$. For $u\in {\rm Lip(\X)}$ we denote by $\Lip(u)$ the Lipschitz constant of $u$. 
The \emph{local Lipschitz constant} $\lip u \colon \X \to \mathbb{R}$ of a function $u\colon \X \to \R$ as
$$\lip u (x):=\lims_{y \to x} \frac{|u(y)-u(x)|}{\sfd(y,x)} = \lim_{\delta \to 0} \sup_{B_\delta(x)} \frac{|u(y)-u(x)|}{\sfd(y,x)} $$ 
with the convention that $\lip u (x)=0$ if $x$ is an isolated point. The function ${\rm sl}(x,\delta) := \sup_{B_\delta(x)} \frac{|u(y)-u(x)|}{\sfd(y,x)}$ is called the slope of $u$ at $x$ at scale $\delta$.
\vspace{2mm}

\noindent
Given a set $A\subseteq \X$, we define by ${\rm int}(A)$ its topological interior and ${\rm ext}(A):={\rm int}(A^c)$.
Given a $\mm$-measurable set $A\subseteq \X$ we define its {\em measure theoretic interior} $I_A$ and its {\em measure theoretic exterior} $O_A$ as follows:
$$I_A := \left\lbrace x\in \X : \lim_{r\to 0} \frac{\mm(B_r(x)\cap A)}{\mm(B_r(x))} = 1 \right\rbrace,\quad O_A := \left\lbrace x\in \X : \lim_{r\to 0} \frac{\mm(B_r(x)\setminus A)}{\mm(B_r(x))} = 1 \right\rbrace.$$
The set $\X \setminus (I_A \cup O_A)$ is the {\em essential boundary} of $A$, denoted by $\partial^e A$. Observe that
$$\partial^e A = \left\{ x\in \X : \lims_{r\to 0}\frac{\mm(B_r(x)\cap A)}{\mm(B_r(x))} > 0 \text{ and } \lims_{r\to 0}\frac{\mm(B_r(x)\setminus A)}{\mm(B_r(x))} > 0  \right\}.$$
It is straightforward consequence of the definitions that $I_A\supseteq {\rm int}(A)$ and $O_A \supseteq {\rm ext}(A)$, so $\partial^e A \subseteq \partial A$.

\subsection{Curves in metric spaces}

Let $(\X,\sfd)$ be a metric space. The space of continuous curves $\gamma \colon [0,1]\to \X$ is denoted by $C([0,1],\X)$. The length of a curve $\gamma \in C([0,1],\X)$ is
\begin{equation*}
    \ell(\gamma):= \sup \left\{ \sum_{i=0}^{N-1} \sfd(\gamma_{t_i}, \gamma_{t_{i+1}}), \,\{ t_i\}_{i=1}^N\subseteq [0,1],\,N \in \mathbb{N},\,
    0=t_0 < t_1 <\dots < t_N =1\right\}.
\end{equation*}
A curve $\gamma \in C([0,1],\X)$ is rectifiable if $\ell(\gamma)<\infty$. 
The metric speed of a curve $\gamma$ at a point $t\in [0,1]$ is  $|\dot{\gamma}_t|:=\lim_{h \to 0}\frac{\sfd(\gamma_{t+h},\gamma_{t})}{|h|} \in [0,\infty]$ if such a limit exists.
A curve $\gamma$ is absolutely continuous, and we write $\gamma \in {\rm AC}([0,1],\X)$, if there exists $0 \le h \in L^1(0,1)$ such that $\sfd(\gamma_1,\gamma_0) \le \int_0^1 h(t) \,\d t$. If $\gamma \in {\rm AC}([0,1],\X)$, $|\dot{\gamma}_t|$ exists and it is finite for a.e.\ $t \in [0,1]$ and $(t \mapsto |\dot{\gamma}_t|) \in L^1([0,1])$. Moreover, such a function is the minimal $h$ (in the a.e.\ sense) that can be chosen in the definition of ${\rm AC}([0,1],\X)$.\\
The integral of a Borel function $g \colon \X \to [0,\infty)$ over $\gamma \in {\rm AC}([0,1],\X)$, is defined as
\begin{equation*}
    \int_\gamma g\,\d s := \int_0^1 g(\gamma_t)\,|\dot{\gamma}_t|\,\d t.
\end{equation*}

\noindent Let us fix two points $x,y \in \X$ and $L \ge 1$. We set
\begin{equation*}
\Gamma_{x,y}^L:=\{ \gamma \in {\rm AC}([0,1],\X),\,\gamma_0=x,\,\gamma_1=y,\,\ell(\gamma)\le L\,\sfd(x,y) \}.
\end{equation*}
An element of $\Gamma_{x,y}^L$ is called a $L$-quasigeodesic joining $x$ to $y$. A curve $\gamma \in C([0,1],\X)$ is a geodesic if $\ell(\gamma)=\sfd(x,y)$, i.e.\ it is a $1$-quasigeodesic joining its endpoints.\\
The metric space $(\X,\sfd)$ is $L$-quasiconvex if $\Gamma_{x,y}^L \neq \emptyset$ for every $x,y \in\X$ and it is geodesic if it is $1$-quasiconvex.\\
Let $u\colon \X \to \R$ be a function. A function $g\colon \X \to [0,+\infty]$ such that
$$\vert u(\gamma_1) - u(\gamma_0)\vert \leq \int_\gamma g\,\d s$$
for every rectifiable curve $\gamma$ is called an \emph{upper gradient} of $u$. The set of upper gradients of $u$ is denoted by ${\rm UG}(u)$.

\subsection{Modulus and approximate modulus of a family of curves}

We recall the notion of $p$-modulus and $p$-approximate modulus of a family of curves with respect a Borel nonnegative locally finite measure $p$.
Let $(\X,\sfd,\mm)$ be a metric measure space, let $p \ge 1$ and let $\Gamma\subseteq C([0,1]),\X)$. The set of admissible densities for $\Gamma$ is 
$${\rm Adm}(\Gamma):=\left\{ \rho\colon \X \to [0,+\infty] : g\text{ Borel and }\int_\gamma \rho\,\d s \ge 1 \text{ for all } \gamma \in \Gamma\right\}.$$
The $p$-modulus of $\Gamma$ with respect to the measure $\mm$ is 
\begin{equation*}
   {\rm Mod}_p(\Gamma,\mm):=\inf_{\rho \in {\rm Adm}(\Gamma)} \int \rho^p\,\d \mm.
\end{equation*}

In this paper we will deal largely with the approximate modulus. In order to define it we introduce the set of admissible sequences of densities for $\Gamma$ as
$${\text{Adm-seq}}(\Gamma):=\left\{ \{\rho_j\}_j:\,\rho_j\colon \X \to [0,+\infty] \text{ Borel and }\limi_{j \to \infty}\int_\gamma \rho_j\,\d s \ge 1 \text{ for all } \gamma \in \Gamma\right\}.$$
The approximate $p$-modulus of $\Gamma$ with respect to the measure $\mm$ is 
\begin{equation*}
   {\rm AM}_p(\Gamma,\mm):=\inf_{\{\rho_j\} \in \text{Adm-seq}(\Gamma)} \limi_{j\to +\infty}\int \rho_j^p\,\d \mm
\end{equation*}

\noindent It is clear from the definition that ${\rm AM}_p(\Gamma,\mm) \leq {\rm Mod}_p(\Gamma,\mm)$ for every family $\Gamma$. Moreover for $p >1$, ${\rm AM}_p(\Gamma,\mm)={\rm Mod}_p(\Gamma,\mm)$ for every $\Gamma \subseteq C([0,1],\X)$ (cp. for instance \cite[Theorem 1]{honzlova2017modulus}). The equality is generally false for $p=1$. Both the modulus and the approximate modulus are outer measures on $C([0,1],\X)$ (cp. \cite[Theorem 15 and 16]{honzlova2017modulus}).

\subsection{The Poincaré inequality and the Riesz potential}
A metric measure space $(\X,\sfd,\mm)$ is doubling if there exists $C_D >0$ such that for every $x \in \X$ and $r >0$ we have
\begin{equation*}
\mm(B_{2r}(x)) \le C_D\, \mm(B_r(x)).
\end{equation*} 
A consequence of the definition of the doubling assumption is that $(\X,\sfd)$ is proper, i.e.\ closed and bounded sets are compact. A metric measure space $(\X,\sfd,\mm)$ is asymptotically doubling if 
$$\lims_{r\to 0} \frac{\mm(B_{2r}(x))}{\mm(B_r(x))} < \infty$$
for $\mm$-a.e.$x\in \X$. We say it is $C$-asymptotically doubling if 
$$\lims_{r\to 0} \frac{\mm(B_{2r}(x))}{\mm(B_r(x))} \leq C$$
for $\mm$-a.e.$x\in \X$.\\
We say that $(\X,\sfd,\mm)$ is $s$-Ahlfors regular with $s >0$ if there exists a constant $C_{\rm AR} >0$ such that for every $x \in \X$ and every $0 < r < {\rm diam}(\X)$ we have
\begin{equation}
    \label{eq:Ahlfors_regular}
    C_{\rm AR}^{-1}\, r^s \le \mm(B_r(x)) \le C_{\rm AR}\, r^s.
\end{equation}
A metric measure space $(\X,\sfd,\mm)$ satisfies a (weak) $p$-Poincar\'{e} inequality if there exists $C_{P} \ge 1$ and $\lambda \ge 1$ such that 
\begin{equation*}
    \dashint_{B_r(x)} \left|u - \dashint_{B_r(x)}u\,\d \mm\right|\,\d \mm \le C_P r\,\left( \dashint_{B_{\lambda r}(x)} (\lip u)^p\,\d \mm \right)^{\frac{1}{p}}
\end{equation*}
for every $u \in {\rm Lip}(\X)$.
%
%By Keith's result \cite{Kei03} or applying density in energy of Lipschitz functions \cite{AmbrosioGigliSavare11-3}, it is equivalently possible to formulate the above inequality for $u \in W^{1,p}(\X)$.
%
A metric measure space $(\X,\sfd,\mm)$ is called a $p$-PI space if it is doubling and it satisfies a $p$-Poincaré inequality. The \emph{structural constants} of a $p$-PI spaces are $C_D,C_P,\lambda$.

\begin{remark}
\label{rem:bilipschitz_equivalence}
If $(\X,\sfd,\mm)$ is a $p$-PI space, then, as proved in \cite[Theorem 17.1]{Cheeger99} (see also \cite[Theorem 8.3.2]{HKST15}), the metric space $(\X,\sfd)$ is $L$-quasigeodesic for some $L \ge 1$ depending only on the structural constants. By considering $\sfd'(x,y):= \inf \{\, \ell(\gamma):\, \gamma_0 = x, \, \gamma_1 = y \,\}$, it is straightforward to check that
\begin{equation}
    \sfd(x,y) \le \sfd'(x,y) \le L \sfd(x,y),
\end{equation}
and that $(\X,\sfd')$ is geodesic. Moreover $(\X,\sfd',\mm)$ is again a $p$-PI space.
\end{remark}

The Poincaré inequality on a doubling metric measure space has several equivalent definitions (cp. \cite[Theorem 3.4]{CaputoCavallucci2024}). In order to state some of them we recall one of the most important objects of this paper: the Riesz potential. Let $(\X,\sfd,\mm)$ be a metric measure space. Given $x,y\in \X$, the \emph{Riesz potential with poles at $x$ and $y$} $R_{x,y} \colon \X \to [0,\infty)$ is defined for every $z \in \X\setminus \lbrace x,y \rbrace$ as
\begin{equation}
\begin{aligned}
    R_{x,y}(z):&= \frac{\sfd(x,z)}{\mm(B_{\sfd(x,z)}(x))} + \frac{\sfd(y,z)}{\mm(B_{\sfd(y,z)}(y))}=: R_x(z)+R_y(z).
\end{aligned}
\end{equation}
Moreover, we define $R_{x,y}(x) = R_{x,y}(y) = 0$. For $L\geq 1$ we set $B_{x,y}^L := B_{2L\sfd(x,y)}(x)\cup B_{2L\sfd(x,y)}(y)$ and $\overline{B}_{x,y}^L := \overline{B}_{2L\sfd(x,y)}(x)\cup \overline{B}_{2L\sfd(x,y)}(y)$. The $L$-\emph{truncated Riesz potential with poles at $x,y$} is
\begin{equation}
    R_{x,y}^L(z):= \chi_{B_{x,y}^L}(z) R_{x,y}(z)
\end{equation}
for every $z \in \X\setminus \lbrace x,y \rbrace$. The corresponding Riesz measure is defined as
\begin{equation}
    \mm_{x,y}^L = R^L_{x,y}\,\mm. 
\end{equation}
It is a measure on $\X$ which is supported on $\overline{B}_{x,y}^L$. This measure has been already studied for instance in \cite{Hei01} and \cite{Kei03}. We can now state two equivalent formulations of the Poincaré inequality on doubling metric measure spaces.
\begin{proposition}[{\cite[Theorem 9.5]{Hei01}}, {\cite[Theorem 3.4]{CaputoCavallucci2024}}]
\label{prop:Poincaré_equivalences_pencil}
    Let $(\X,\sfd,\mm)$ be a doubling metric measure space. The following are quantitatively equivalent:
    \begin{itemize}
        \item[($p$-PI)] $(\X,\sfd,\mm)$ satisfies a $p$-Poincar\'{e} inequality;
        \item[($p$-PtPI)] $\exists C > 0$, $L\geq 1$ such that for every $x,y \in \X$ and every $u\in \Lip(\X)$ it holds
        \begin{equation*}
        \label{eq:Riesz_PtPI}
            |u(x)-u(y)|^p\le C\, \sfd(x,y)^{p-1}\, \int_\X (\lip u)^p\,\d \mm_{x,y}^{L}.
        \end{equation*} 
        \item[($p$-pencil)] $\exists C > 0$, $L\geq 1$ such that for every $x,y \in \X$ there exists $\alpha\in \mathcal{P}(\Gamma_{x,y}^L)$ such that
        $$\left( \int \int_\gamma g \,\d s \d\alpha \right)^p \leq C\,\sfd(x,y)^{p-1} \int_\X g^p\,\d\mm_{x,y}^L$$
        for every nonnegative Borel function $g$.
    \end{itemize}
\end{proposition}

The next result establishes some basic properties of the Riesz measures.
\begin{proposition}
\label{prop:properties_mxy}
    Let $(\X,\sfd,\mm)$ be a $C_D$-doubling metric measure space and fix $x,y\in \X$ and $L\geq 1$. Then
    \begin{itemize}
        \item[(i)] $\mm_{x,y}^L(\X) \leq 8C_DL\sfd(x,y)$. In particular $\mm_{x,y}^L$ is a finite Borel measure;
        \item[(ii)] $\mm_{x,y}^L$ is $C_D$-asymptotically doubling, in particular $(\X,\sfd,\mm_{x,y}^L)$ is a Vitali space. 
    \end{itemize}
\end{proposition}
For the definition of Vitali space we refer to \cite[§3.4]{HKST15}.
\begin{proof}
    The first estimate follows by 
    $$\mm_{x,y}^L(\X) \leq \int_{B_{4L\sfd(x,y)}(x)} R_x \,\d\mm + \int_{B_{4L\sfd(x,y)}(y)} R_y \,\d\mm  \leq 8C_DL\sfd(x,y),$$
    where for the last inequality we rewrited the Riesz potential in dyadic scales as in \cite[p. 72]{Hei01}. Observe that the poles $x,y$ could have positive $\mm$-measure, so it is important to have defined $R_{x,y}(x) = R_{x,y}(y) = 0$ for this computation.\\
    For the second point we recall that $(\X,\sfd,\mm)$ is a Vitali space since $\mm$ is doubling (cp. \cite[Theorem 3.4.3]{HKST15}). Moreover $R_{x,y}^L \in L^1(\mm)$, as the computation of the first item shows. Therefore, by \cite[p.77]{HKST15}, for $\mm$-a.e.$z$ we have that 
    $$\lim_{r\to 0}\dashint_{B(z,r)} R_{x,y}^L \,\d\mm = R_{x,y}^L(z).$$
    For these points we have
    $$\lims_{r\to 0} \frac{\mm_{x,y}^L(B_{2r}(z))}{\mm_{x,y}^L(B_r(z))} \leq C_D \lims_{r\to 0} \frac{\dashint_{B_{2r}(z)}R_{x,y}^L\,\d\mm}{\dashint_{B_r(z)}R_{x,y}^L\,\d\mm} = C_D.$$
    Since $\mm_{x,y}^L$ is absolutely continuous with respect to $\mm$ it follows that for $\mm_{x,y}^L$-a.e. point the estimate $\lims_{r\to 0} \frac{\mm_{x,y}^L(B_{2r}(z))}{\mm_{x,y}^L(B_r(z))} \leq C_D$ holds, proving that $(\X,\sfd,\mm_{x,y}^L)$ is asymptotically doubling. Then it is a Vitali space by \cite[Theorem 3.4.3]{HKST15}.
\end{proof}
The next continuity result for the Riesz potential will be used intensively along the paper.
\begin{lemma}
\label{lemma:continuity_riesz_potential}
    If $(\X,\sfd,\mm)$ is a geodesic, doubling metric measure space then the map
    $$\mathcal{D}(R) \to [0,+\infty], \quad (x,y,z) \mapsto R_{x,y}(z)$$
    is continuous, where $\mathcal{D}(R) := \lbrace (x,y,z)\in \X : z\neq x \neq y \neq z \rbrace$.
\end{lemma}
\begin{proof}
    Let us take sequences $x_j$, $y_j$ and $z_j$ converging respectively to $x,y,z$ with the assumption that $z_j\neq x_j \neq y_j \neq z_j$ and $z\neq x \neq y \neq z$. In particular $\min\lbrace \sfd(x,z), \sfd(y,z)\rbrace > 0$. Therefore we can apply \cite[Proposition 11.5.3]{HKST15} to get that $\mm(B_{\sfd(x_j,z_j)}(x_j))$ converges to $\mm(B_{\sfd(x,z)}(x))$, and similarly for $\mm(B_{\sfd(y_j,z_j)}(y_j))$. This shows that $R_{x_j,y_j}(z_j)$ converges to $R_{x,y}(z)$.
\end{proof}

\section{Several energies to measure boundaries}
\label{sec:several_energies}
In this section we introduce several ways to measure the energy of boundaries of separating sets appearing in Theorem \ref{theo:main-intro-p=1}, where the separating sets have been defined in Definition \ref{def:separating_sets}.
The first one is the perimeter, which is by now a classical object of investigation in metric measure geometry (see \cite{Amb02}, \cite{AmbrosioDiMarino14}, \cite{Lahti2020} for basics and fine properties and for instance \cite{KorteLahtiLiShan}, \cite{CaputoRajala}, \cite{BrenaNobiliPasqualetto} for related recent problems in analysis on metric spaces). \\
Let $(\X,\sfd,\mm)$ be a metric measure space. The
{\em total variation} of $u \in L^1_\loc (X)$ on an open set $A \subset X$ is
\[ 
|D u|(A) := \inf \left\{ \liminf_{i\to \infty} \int_A \lip u_i \,\d\mm :\, u_i \in \Lip_\loc(A),\, u_i \to u \mbox{ in } L^1_\loc(A) \right\}. 
\]
%A function $u$ is said to be of {\em bounded variation} if $|D u|(X)$ is finite.  
%Let $BV(\X):=\{ u \in L^1(\mm):\, |D u|(\X) <\infty \}$ denote the set of functions of bounded variation.
For an arbitrary set $B\subset X$, we define
\[
|D u|(B):=\inf\{|D u|(A):\, B\subset A,\,A\subset X
\text{ open}\}.
\]
The set function $|D u|\colon \mathscr{B}(\X) \to [0,\infty]$ is a
Borel measure on $\X$ by \cite[Theorem 3.4]{Mir03} (see also \cite[Lemma 5.2]{AmbrosioDiMarino14}).
We say that a Borel set $E \subseteq \X$ has finite perimeter if $|D\chi_E|(\X) < \infty$. In this case, we set ${\rm Per}(E,\cdot):=|D \chi_E|$.

%Let $E\subseteq \X$ be Borel and $A\subseteq \X$ be open. The \emph{perimeter} of $E$ in $A$ is
%\begin{equation}
%    \label{eq:perimeter_defin_open_sets}
%    {\rm Per}(E,A) := \inf \left\{ \limi_{i\to \infty} \int_A \lip u_i \,\d\mm :\, u_i \in \Lip_\loc(A),\, u_i \to \chi_E \mbox{ in } L^1(A) \right\}.
%\end{equation}
%A set $E$ is said of finite perimeter if ${\rm Per}(E,\X) < \infty$. For an arbitrary set $B\subseteq \X$ we define the perimeter of $E$ in $B$ as
%$${\rm Per}(E, B):=\inf\{{\rm Per}(E,A):\, B\subset A,\,A\subset X
%\text{ is open}\}.$$
%The assignment $B\mapsto {\rm Per}(E,\cdot)$ defines a Radon measure on $\X$ as proven in \cite[Lemma 5.2]{AmbrosioDiMarino14}.
\noindent It was proved in \cite[Theorem 1.1]{AmbrosioDiMarino14} that
\begin{equation}
    \label{eq:BV_upper_gradients}
    {\rm Per}(E,A) = \inf \left\{ \limi_{i\to \infty} \int_A g_i \,d\mm :\,u_i\in L^1(A),\, g_i \in {\rm UG}(u_i),\, u_i \to \chi_E \mbox{ in } L^1(A) \right\}
\end{equation}
for all $A\subseteq \X$ open and bounded. Here the boundedness assumption is enough for having that $L^1_\loc(A)$-convergence coincides with $L^1(A)$-convergence, the latter being the one considered in the right hand side term above and in the equivalences in \cite{AmbrosioDiMarino14}. 
Whenever we are interested in stressing the measure on the metric measure space $(\X,\sfd,\mm)$, we will write ${\rm Per}_{\mm}(E, \cdot)$.
The perimeter measure satisfies a coarea inequality.
\begin{lemma}
    Let $(\X,\sfd,\mm)$ be a metric measure space. If $u \in L^1_\loc(\X) \cap {\rm Lip}(\X)$ then
\begin{equation}
\label{eq:coarea_inequality_BV}
    \int_{-\infty}^{+\infty} \int g \,\d {\rm Per}(\{ u >t\}, \cdot)\,\d t \le \int_\X g \,\lip u\,\d \mm
\end{equation}
for every Borel $g\geq 0$.
\end{lemma}
\begin{proof}
 The proof follows by the coarea formula for BV functions (see \cite{Mir03}) and that, since $u \in L^1_{{\rm loc}}(\X) \cap {\rm Lip}(\X)$, $\lip u \in {\rm UG}(u)$ and $|D u| \le \lip u\,\mm$.
\end{proof}
In particular, applying the coarea inequality to the Lipschitz function $\sfd(x,\cdot)$, we get the following well known fact.
\begin{lemma}
\label{lemma:balls_have_finite_perimeter}
    Let $(\X,\sfd,\mm)$ be a metric measure space and let $x\in \X$. Then the ball $B_r(x)$ has finite perimeter for almost every $r>0$.
\end{lemma}

We end this section with a list of useful well known properties of the perimeter.
\begin{remark}
\label{rem:properties_perimeter}
Let $(\X,\sfd,\mm)$ be a metric measure space. Let $E,F\subseteq \X$ be Borel sets.
Let $U \subseteq \X$ be open and let $B \subseteq \X$ be Borel. Then:
\begin{itemize}
    \item[(i)] \emph{Locality}. If $\mm((E \Delta F) \cap U)=0$, then 
    \begin{equation}
    \label{eq:locality_perimeter}
        {\rm Per}(E,U) = {\rm Per}(F,U);
    \end{equation}
    %\item[ii)] \emph{Lower semicontinuity}. For every open set $U \subseteq \X$, the function $E \mapsto {\rm Per}(E,U)$ is lower semicontinuous with respect to $L^1(\mm)$-topology, namely if $\chi_{E_n} \to \chi_E$ in $L^1(\mm)$, then ${\rm Per}(E,U) \le \limi_n {\rm Per}(E_n,U)$;
    \item[(ii)] \emph{Subadditivity}. It holds ${\rm Per}(E \cup F,B) + {\rm Per}(E \cap F,B) \le {\rm Per}(E,B) + {\rm Per}(F,B)$;
    \item[(iii)] \emph{Complementation}. It holds ${\rm Per}(E ,B) = {\rm Per}(\X \setminus E,B)$;
    \item[(iv)] \emph{Support}. It holds ${\rm spt}({\rm Per}(E,\cdot)) \subseteq \partial E$.
\end{itemize}
Let us comment the proof of (iv). The sequence $u_j = u$, where $u=1$ on ${\rm int}(E)$ and $u=0$ on ${\rm ext}(E)$ is admissible for the computation of ${\rm Per}(E,{\rm int}(E)\cup {\rm ext}(E))$, thus giving ${\rm Per}(E,{\rm int}(E)\cup {\rm ext}(E))=0$.
\end{remark}

%It was proved in \cite[Thm.\ 1.1]{AmbrosioDiMarino14}, given $u \in L^1(\mm)$, its total variation can be computed, for an open set
%\begin{equation}
%    \label{def:BV_upper_gradients}
%    |D u|(A) := \inf \left\{ \liminf_{i\to \infty} \int_A g_i \,d\mm :\, g_i \in {\rm UG}(u_i),\, u_i \to u \mbox{ in } L^1_\loc(A) \right\}. 
%\end{equation}
%We denote by ${\rm UG}(u)$ the set of upper gradients of $u$ and by ${\rm WUG}_p(u)$ the set of $p$-weak upper gradients of $u$ (see \cite{HKST15}).
%We recall that, since ${\rm WUG}_p(u)=\overline{ {\rm UG}(u)}^{L^p}$, we can replace $g_i \in {\rm UG}(u_i)$ with $g_i \in {\rm WUG}_p(u_i)$.

The next one is the codimension-$p$ Hausdorff measure. For every $p \geq 0$, we define the gauge function
\begin{equation}
    h_p(B_r(x)):=\frac{\mm(B_r(x))}{r^p}
\end{equation}
for every $x \in \X$ and $r >0$.
We denote by $\codH{p}_\delta$ the pre-measure with parameter $\delta >0$ defined as
\begin{equation}
    \codH{p}_\delta(E):=\inf \left\{ \sum_{j=1}^\infty \frac{\mm(B_{r_j}(x_j))}{r_j^p}:\,E \subset \bigcup_{j=1}^\infty B_{r_j}(x_j),\,\sup_j r_j <\delta \right\}.
\end{equation}

\begin{definition}
    The codimension-$p$ Hausdorff measure $\codH{p}$ is the Borel regular outer measure defined as
\begin{equation}
    \codH{p}(E) :=\sup_{\delta >0} \codH{p}_\delta(E) \,\text{for every }E\subseteq \X.
\end{equation}
When we want to emphasize the measure with respect to with we are computing the codimension $p$-Hausdorff measure we write $\codH{p}_\mm$.
\end{definition}

\begin{remark}
\label{remark:hcod-p_properties}
    The following three properties follow easily from the definition.
    \begin{itemize}
        \item[(i)] If $(\X,\sfd,\mm)$ is $s$-Ahlfors regular then we have $C_{\rm AR}^{-1} \mathcal{H}^{s-p} \le \codH{p} \le C_{\rm AR} \mathcal{H}^{s-p}$, where $\mathcal{H}^{s-p}$ is the $(s-p)$-dimensional spherical Hausdorff measure.
        \item[(ii)] If $(\X,\sfd,\mm)$ is a doubling metric measure space then $\codH{0} = \mm$. 
        %{\color{red}Indeed, by definition $\mm(E) \le \codH{0}(E)$ for every Borel set $E$. To conclude, by an application of monotone class theorem, it suffices to prove that $\codH{0}(U) \le \mm(U)$ for every open set $U$. Fix $\delta > 0$, a dense set $\{x_i\}_i \subseteq U$ and consider $\mathcal{B}:=\{ \overline{B}_r(x_i): \overline{B}_r(x_i) \subseteq U,\, r <\delta, r\in \Q\}$. This is a fine covering of $U$. By the Vitali property we can find a subfamily of pairwise disjoint balls $\{ B_{r_j}(x_j) \}_j$ such that $\mm(U \setminus \bigcup_j B_{r_j}(x_j))=0$. Thus $\codH{0}_{\delta}(U) \le \mm(U)$. Taking the limit as $\delta$ converging to $0$ we conclude. }
        
        \item[(iii)] The codimension-$p$ Hausdorff measures behave like the classical Hausdorff measures in the following sense: if $\codH{p}(E) > 0$ then $\codH{q}(E) = +\infty$ for all $q>p$, while if $\codH{p}(E) < + \infty$ then $\codH{q}(E) = 0$ for all $q<p$.
    \end{itemize}
\end{remark}
We introduce the following proposition, whose proof is a modification of standard arguments for the Hausdorff measure.
\begin{proposition}
\label{prop:comparison_finiteness_Hcod_mm}
    Let $(\X,\sfd,\mm)$ be a metric measure space and let $p >0$. Let $A\subseteq \X$ be Borel. If $\codH{p}_{\mm}(A) < \infty$, then $\mm(A)=0$.
\end{proposition}
\begin{proof}
    Assume $\mm(A)>0$. For every covering $A \subseteq \bigcup_{i=1}^\infty B_{r_i}(x_i)$ with $r_i< \delta$ we have $\sum_i \frac{\mm(B_{r_i}(x_i))}{r_i^p} \ge \frac{1}{\delta^p}\,\sum_i \mm(B_{r_i}(x_i))\ge \frac{1}{\delta^p} \,\mm(A)$.
    This implies that $\codH{p}_{\mm,\delta}(A) \ge \frac{1}{\delta^p} \mm(A)$. Thus, by considering the limit as $\delta \to 0$, we get $\codH{p}_{\mm}(A)=\infty$.
\end{proof}
When $p=1$ we have a coarea inequality for the codimension-$1$ Hausdorff measure. It was studied in \cite[Proposition 5.1]{AmbDiMarGig17}, where they state it for \emph{doubling} metric measure space. Moreover in \cite{AmbDiMarGig17} the result is stated for $g=\chi_B$, $B\subseteq \X$ Borel. It can be extended to the following version by the classical argument of approximating a nonnegative Borel function via an increasing sequence of simple functions. We point out that the result was stated with the addition assumption that $\mm(\{ f >0 \}) <\infty$, but this hypothesis can be removed via a cut-off argument.
Moreover, we also introduce a coarea formula which is suited for $\mm_{x,y}^L$, for which we provide the proof adapting the arguments in \cite[Proposition 5.1]{AmbDiMarGig17}.
\begin{proposition}[Coarea inequality for codimension-$1$ Hausdorff measure] 
\label{prop:coarea_inequality_hausdorffmeasure}
Assume that $(\X,\sfd,\mm)$ is a doubling metric measure space.
Let $0\le u \in {\rm Lip}(\X)$. Then for every $g\colon \X \to [0,+\infty]$ Borel one has
\begin{equation}
\label{eq:coareain2}
\int_0^\infty\int_{\lbrace u=t \rbrace} g\,\d\codH{1}_\mm\,\d t\leq  4 \int_\X  g\,\lip u \,\d \mm.
\end{equation}
Moreover if $x,y\in \X$ and $L\geq 1$, we have
\begin{equation}
    \label{eq:coareain2_mxyL}
\int_0^\infty\int_{\lbrace u=t \rbrace} g\,\d\codH{1}_{\mm_{x,y}^L}\,\d t\leq  4 \int_\X  g\,\lip u \,\d \mm_{x,y}^L.
\end{equation}
\end{proposition}
\begin{proof}
    Inequality \eqref{eq:coareain2} is exactly \cite[Proposition 5.1]{AmbDiMarGig17}, where they use the gauge function $h_1/2$ in order to define $\codH{1}$. That is why we have an additional $2$ factor. We focus on \eqref{eq:coareain2_mxyL}.
    \vspace{1mm}
    
    \noindent{\bf Step 1}. We claim that, given a Borel set $B \subseteq \X$, $\mm_{x,y}^L(B)=0$ implies $\codH{1}_{\mm_{x,y}^L}(B \cap \{u=t\})=0$ for a.e.\ $t$. Observe that it is sufficient to consider the case $x,y\notin B$. 
    Indeed for a generic $B$ we set $B' = B\setminus \lbrace x,y \rbrace$. Then $\mm_{x,y}^L(B) = \mm_{x,y}^L(B')$ by our definition of $\mm_{x,y}^L$. Moreover $\lbrace x,y \rbrace \cap \{u = t\}$ is not empty for at most two values of $t$.
    So, if $x,y\notin B$ and $\mm_{x,y}^L(B) = 0$ then $\mm(B \cap B_{x,y}^L)=0$, thus by \eqref{eq:coareain2} with $g=\chi_B$ we get that $\codH{1}_{\mm}(B \cap B_{x,y}^L \cap \{u=t\})=0$ for a.e.\ $t$. For every such $t \neq u(x),u(y)$ we have $\sfd(x,\{u=t\})>0$ and $\sfd(y,\{u=t\})>0$ since $u$ is Lipschitz. Thus, $R_{x,y}^L$ is both bounded from above and below on $\{u=t\} \cap B_{x,y}^L$, which gives, by the very definition of codimension Hausdorff measure, $\codH{1}_{\mm_{x,y}^L}(B \cap \{u=t\})=0$ for a.e.\ $t$.
    \vspace{1mm}

    \noindent{\bf Step 2.}
    Let $\delta >0$, let $U\subseteq \X$ be open and bounded, $S \subseteq U$ be countable dense and consider the fine covering $\mathcal{F}:=\{ B_r(x), r \in \Q, \, x \in S,\, r< \min ( \sfd(x,U^c),\delta)\}$.
    Since $\mm_{x,y}^L$ is a Vitali space (Proposition \ref{prop:properties_mxy}) we find a disjoint subfamily $\{ B_{r_i}(x_i)\}$ such that $\mm_{x,y}^L(U \setminus \bigcup_i B_{r_i}(x_i))=0$. Call $\tilde{B}:=\bigcup_i B_{r_i}(x_i)$. Then
    \begin{equation}
        \codH{1}_{\mm_{x,y}^L,\delta}(U \cap \{ u =t \}) \le \codH{1}_{\mm_{x,y}^L,\delta}(\tilde{B} \cap \{ u =t \}) + \codH{1}_{\mm_{x,y}^L,\delta}((U \setminus \tilde{B})\cap \{ u =t \}) \le \codH{1}_{\mm_{x,y}^L,\delta}(\tilde{B} \cap \{ u =t \})
    \end{equation}
    for a.e.\ $t$, by using Step 1. Thus $\codH{1}_{\mm_{x,y}^L,\delta}(U \cap \{ u =t \}) =\codH{1}_{\mm_{x,y}^L,\delta}(\tilde{B} \cap \{ u =t \})$. Define $t_i^-:=\inf_{B_{r_i}(x_i)} u$ and $t_i^+:=\sup_{B_{r_i}(x_i)} u$. If $B_{r_i}(x_i) \cap \{ u=t\} \neq 0$, then $t \in [t_i^-, t_i^+]$; thus, 
    \begin{equation*}
        \codH{1}_{\mm_{x,y}^L,\delta}(\tilde{B}\cap \{ u=t\}) \le \sum_{i,\, \text{s.t. } t \in [t_i^-,t_i^+]}\frac{\mm_{x,y}^L(B_{r_i}(x_i))}{r_i}.
    \end{equation*}
    Integrating betweeen $0$ and $\infty$ we obtain
    \begin{equation}
        \int_0^\infty \codH{1}_{\mm_{x,y}^L,\delta}(\tilde{B}\cap \{ u=t\}) \le \sum_i \frac{\mm_{x,y}^L(B_{r_i}(x_i))}{r_i} (t_i^+ - t_i^-) \le 4\int_{\tilde{B}} {\rm sl}(x,\delta) \,\d \mm_{x,y}^L,      
    \end{equation}
    where in the last inequality we used for $x \in B_{r_i}(x_i)$ 
    \begin{equation*}
        \frac{t_i^+-t_i^-}{r_i}= \sup_{y,y' \in B_{r_i}(x_i)} \frac{f(y)-f(y')}{r_i} \le 4 \sup_{y \in B_{r_i}(x_i)} \frac{|f(y)-f(x)|}{\sfd(x,y)} \le 4 {\rm sl}(x,\delta).
    \end{equation*}
    Therefore
    \begin{equation*}
        \int_0^\infty \codH{1}_{\mm_{x,y}^L,\delta}(U\cap \{ u=t\}) \le 4\int_{U} {\rm sl}(x,\delta) \,\d \mm_{x,y}^L
    \end{equation*}
    and taking the limit as $\delta \to 0$, we obtain $\int_0^\infty \codH{1}_{\mm_{x,y}^L}(U\cap \{ u=t\}) \le \int_{U} \lip u \,\d \mm$. By monotone class theorem, we extend the previous inequality to all Borel sets and monotone approximation to all Borel nonnnegative functions $g$, thus proving \eqref{eq:coareain2_mxyL}.
\end{proof}
\begin{remark}
    It is not clear to us if a coarea-formula holds under the only assumption that $(\X,\sfd,\mm)$ is a Vitali space or even with the stronger assumption that it is asymptotically doubling. For this reason we adapted the proof of \cite{AmbDiMarGig17} for specific case of $\mm_{x,y}^L$.
\end{remark}

We remark that some version of the coarea inequality for $p>1$ cannot hold because the $\codH{p}$ measure of the level sets of a Lipschitz function is generally $+\infty$.

\noindent A third way to measure the energy of a set is the Minkowski content.
\begin{definition}
    Let $(\X,\sfd,\mm)$ be a metric measure space and let $A\subseteq \X$ be Borel. The codimension-$p$ Minkowski content of $A \subset \X$ is
\begin{equation*}
    \mm^{+,p}(A):=\limi_{r \to 0} \frac{\mm\left(\overline{B}_r(A) \setminus A \right)}{r^p}.
\end{equation*}
\end{definition}
%
%If $p=1$, we use the shorthand notation $\mm^{-}(A)$ for $\mm^{-,1}(A)$ and $\mm^{+}(A)$ for $\mm^{+,1}(A)$.

The $1$-Minkowski content satisfies a coarea inequality. We write both the versions of \cite[Prop.\ 3.5]{KorteLahti2014} and of \cite[Lemma 3.2]{AmbDiMarGig17}. 
\begin{proposition}[Coarea inequality for Minkowski content]
\label{prop:coarea_inequality_minkowski}
    Let $(\X,\sfd,\mm)$ be a metric measure space and suppose $\mm$ is finite. Then, for every bounded $u \in {\rm Lip}(\X)$, we have
    \begin{equation}
        \label{eq:coareain3_AGD}
        \int_{-\infty}^{\infty} \mm^{+,1}(\{ u \geq t\}) \,\d t \le \int_\X {\rm \lip} u \,\d \mm
    \end{equation}
    and
    \begin{equation}
        \label{eq:coareain3_KL}
        \int_{-\infty}^{\infty} \mm^{+,1}(\partial \{ u > t\}) \,\d t \le 2\int_\X {\rm \lip} u \,\d \mm.
    \end{equation}
\end{proposition}

\begin{proof}
    The first inequality \eqref{eq:coareain3_AGD} is exactly \cite[Lemma 3.2]{AmbDiMarGig17}. For the second one we remark that the version of the Minkowski content used in \cite{KorteLahti2014} is the following:
    $$\mm_{\text{KL}}^+(A) = \limi_{r\to 0} \frac{\mm(B_r(A))}{2r}.$$
    It is clear that $\mm^{+,1}(A) \leq 2\mm_{\text{KL}}^+(A)$, so \eqref{eq:coareain3_KL} follows directly from \cite[Prop.\ 3.5]{KorteLahti2014}.
\end{proof}

The next energy we introduce is the variational $p$-capacity.
\begin{definition}
    Let $(\X,\sfd,\mm)$ be a metric measure space and let $A\subseteq U$ be two subsets of $\X$ such that $U$ is open. Let $p\geq 1$. The variational $p$-capacity of the couple $(A,U)$ (see for instance \cite[§6.3]{BjornBjorn2011}) is
    \begin{equation}
    \label{eq:defin_variational_capacity}
    \textup{cap}_{\mm}^p(A,U) := \inf \left\{ \left(\int g^p\,\d \mm\right)^{\frac{1}{p}},\, g \in \text{UG}(u), \text{ spt}(u) \subset U, u\geq 1 \text{ on } A \right\}.
    \end{equation}
    The $p$-variational capacity of a subset $A\subseteq \X$ with respect to $\mm$ is
\begin{equation*}
    \textup{cap}^p_{\mm}(A) = \sup_{A \subset U, U \text{ open}} \textup{cap}^p_{\mm}(A, U).
\end{equation*}
\end{definition}
There are many possible variants of the variational capacity of a set. We decided for the goals of this manuscript to stick with the one above. By a classical truncation argument it is enough to take functions $u$ such that $\chi_A \leq u \leq 1$ in the right hand side of \eqref{eq:defin_variational_capacity}, see \cite[§6.3, p.161]{BjornBjorn2011}.
\begin{remark}
    \label{remark:capacity_compact_set}
    If $A \subseteq \X$ is compact then every open set $U$ containing $A$ contains also $B_r(A)$ for some small enough $r>0$. From the definition of variational capacity it follows that if $A\subseteq U_1 \subseteq U_2$, with $U_1,U_2$ open, then $\textup{cap}^p_{\mm}(A, U_1) \geq \textup{cap}^p_{\mm}(A, U_2)$. Therefore, if $A\subseteq \X$ is compact, then 
    $$\textup{cap}^p_{\mm}(A) = \sup_{r>0} \textup{cap}^p_{\mm}(A, B_r(A)).$$
\end{remark}
The variational capacity of a compact set is always bounded from above by the lower Minkowski content, while for $p=1$ it is bounded from below by the perimeter.
\begin{lemma}
    \label{lemma:perimeter<capacity<minkowski}
    Let $(\X,\sfd,\mm)$ be a metric measure space and let $A\subseteq \X$ be compact. Then 
    $$\textup{cap}^p_{\mm}(A) \le \mm^{+,p}(A)$$ for all $p\geq 1$. 
    Moreover 
    $${\rm Per}_\mm(A, \X) \leq \textup{cap}^1_{\mm}(A).$$
\end{lemma}

\begin{proof}
    Let us start with the first inequality. By Remark \ref{remark:capacity_compact_set}, it is enough to estimate from above $\textup{cap}^p_{\mm}(A,B_r(A))$. We fix $\tau <1$ and we consider the function $u_\tau =\max(1-\sfd_A(\cdot)/\tau r,0)$. It is straightforward to check that $\lip u_\tau \le 1/\tau r$ on $\overline{B_{r}(A)}\setminus A$ and zero otherwise. Moreover $u_\tau$ has compact support in $B_r(A)$. Thus
    \begin{equation*}
        \textup{cap}^p_{\mm}(A,B_r(A)) \le \int (\lip u_\tau)^p\,\d \mm \le \frac{\mm(\overline{B_r(A)}\setminus A)}{\tau^pr^p}.
    \end{equation*}
    Taking the infimum over $\tau <1$, we get 
    $$\textup{cap}^p_{\mm}(A,B_r(A)) \leq \frac{\mm(\overline{B_r(A)}\setminus A)}{r^p}$$
    for every $r>0$. Taking the liminf as $r \to 0$, we conclude the proof of the first part of the statement.\\
    We now assume that $p=1$. We have $\textup{cap}^1_\mm(A) = \sup_{r>0}\textup{cap}^1_{\mm}(A,B_r(A))$. For every $r > 0$ let $g_r \in \text{UG}(u_r)$, with $u_r \colon \X \to \R$ such that spt$(u_r) \subset B_r(A)$ and $u_r = 1$ on $A$, be such that $\int g_r\,\d\mm \leq \textup{cap}^1_{\mm}(A,B_r(A)) + r$. We can choose $u_r = 1$ on $A$ as we recalled after the definition of the variational capacity. We observe that $u_r$ converges pointwise to $\chi_A$, because $A$ is closed. So $u_r \to \chi_A$ in $L^1(\X)$ by dominate convergence. Using Remark \ref{rem:properties_perimeter}.(iv) we get that ${\rm Per}_\mm(A,\X) = {\rm Per}_\mm(A,B_s(A))$ for every $s>0$. Thus, since $B_s(A)$ is open and bounded for $s>0$, we can use \eqref{eq:BV_upper_gradients} and have that ${\rm Per}_\mm(A,B_s(A)) \le \limi_{r\to 0} \int_\X g_r\,\d\mm$. Thus, we have
    $${\rm Per}_\mm(A,\X) = {\rm Per}_\mm(A,B_s(A)) \leq \limi_{r\to 0} \int_\X g_r\,\d\mm \leq \limi_{r\to 0} \textup{cap}^1_{\mm}(A,B_r(A)) + r = \textup{cap}^1_\mm(A),$$
    concluding the proof.
\end{proof}

The last energies we consider are defined through the approximate modulus.
\begin{definition}
Let $(\X,\sfd,\mm)$ be a metric measure space and let $A\subseteq \X$. We consider the family $\Gamma(A):=\{ \gamma \in C([0,1],\X):\, \gamma \cap A \neq \emptyset \}.$
Let $p \ge 1$. The approximate $p$-modulus of $A$ is 
$${\rm AM}_p(A,\mm):={\rm AM}_p(\Gamma(A),\mm).$$
We also introduce the family $\Gamma(A)^\pitchfork:=\{\gamma \in C([0,1],\X):\, \gamma_0 \in I_A, \gamma_1\in O_A \}.$
The transversal approximate $p$-modulus of $A$ is
$${\rm AM}^\pitchfork_p(A,\mm):={\rm AM}_p(\Gamma(A)^\pitchfork,\mm).$$
\end{definition}

The approximate moduli above can be computed with arbitrarily short curves.
\begin{lemma}
    \label{lemma:AM_localized_near_boundary}
    Let $(\X,\sfd,\mm)$ be a metric measure space, let $A\subseteq \X$ and let $\delta > 0$. Define the following families of curves:
    $$\Gamma(A)_\delta := \lbrace \gamma \in \Gamma(A) : \ell(\gamma)\leq \delta \rbrace \subseteq \Gamma(A);$$
    $$\Gamma(A)^\pitchfork_\delta = \lbrace \gamma \in \Gamma(A)^\pitchfork : \ell(\gamma)\leq \delta \rbrace \subseteq \Gamma(A)^\pitchfork.$$
    Then
    $${\rm AM}_p(A,\mm) = {\rm AM}_p(\Gamma(A)_\delta, \mm), \quad {\rm AM}_p^\pitchfork(A,\mm) = {\rm AM}_p(\Gamma(A)^\pitchfork_\delta, \mm).$$
\end{lemma}
\begin{proof}
    The inclusions $\Gamma(A)_\delta \subseteq \Gamma(A)$ and $\Gamma(A)^\pitchfork_\delta \subseteq \Gamma(A)^\pitchfork$ give ${\rm AM}_p(\Gamma(A)_\delta, \mm) \leq {\rm AM}_p(A,\mm)$ and ${\rm AM}_p(\Gamma(A)^\pitchfork_\delta, \mm) \leq {\rm AM}_p^\pitchfork(A,\mm)$. On the other hand for every $\gamma \in \Gamma(A)$ we can find a subcurve $\gamma_\delta \in \Gamma(A)_\delta$. Therefore an admissible sequence for $\Gamma(A)_\delta$ is admissible also for $\Gamma(A)$. This implies ${\rm AM}_p(A,\mm) \leq {\rm AM}_p(\Gamma(A)_\delta, \mm)$ and similarly ${\rm AM}^\pitchfork_p(A,\mm) \leq {\rm AM}_p(\Gamma(A)^\pitchfork_\delta, \mm)$.
\end{proof}

The approximate modulus is bounded from above by the $p$-Minkowski content when the set has measure zero.
\begin{lemma}
\label{lemma:approximate_modulus_implies_minkowski}
    Let $(\X,\sfd,\mm)$ be a metric measure space and let $p \ge 1$. Let $A\subseteq \X$ be such that $\mm(A)=0$. Then 
    $${\rm AM}_p(A,\mm) \leq \mm^{+,p}(A).$$
\end{lemma}
\begin{proof}
    Let $r_j > 0$ be a generic decreasing sequence converging to $0$. We consider the functions $\rho_j = \frac{1}{r_j}\chi_{B_{r_j}(A)}$. This sequence is admissible for $\Gamma(A)$. Indeed, if $\gamma \in \Gamma(A)$ is not constant, then its image has diameter at least $2r_{j_0}$ for some fixed $j_0$. For every $j\geq j_0$ we have $\ell(\gamma \cap B_{r_j}(A)) \geq r_j$, since $\gamma \cap A \neq \emptyset$. Then $\limi_{j\to + \infty} \int_{\gamma} \rho_j \geq 1$. Therefore
    $${\rm AM}_p(A,\mm) \leq \limi_{j\to +\infty} \int \rho_j^p \,\d\mm = \limi_{j\to +\infty} \frac{\mm(B_{r_j}(A))}{r_j^p}.$$
    By the arbitrariness of the sequence $r_j$ and the fact that $\mm(A) = 0$ we get
    $${\rm AM}_p(A,\mm) \leq \mm^{+,p}(A).$$
\end{proof}
We conclude this section with an easy but useful remark.
\begin{remark}
\label{rem:biLipschitz_change}
Let $\sfd,\sfd'$ be two equivalent metrics on a set $\X$, i.e. there exists a constant $C_0 \geq 1$ such that
\begin{equation}
    C_0^{-1}\sfd \le  \sfd' \le C_0 \sfd,
\end{equation}
and let $\mm$ be a Borel measure on $\X$. Then
\begin{itemize}
    \item[(i)] if $(\X,\sfd,\mm)$ is $C_D$-doubling then $(\X,\sfd',\mm)$ is $C_D^{1+2\log_2C_0}$-doubling.
    \item[(ii)] for every set $A\subseteq \X$ and every $p\geq 1$ it holds
    \begin{equation*}
        \begin{aligned}
         &C_0^{-1}\cdot {\rm Per}_{\mm, \sfd}(A)\le {\rm Per}_{\mm, \sfd'}(A) \le C_0\cdot{\rm Per}_{\mm, \sfd}(A).\\
         &C_0^{-p}\cdot \codH{p}_{\mm, \sfd}(A) \le \codH{p}_{\mm, \sfd'}(A) \le C_0^p\cdot\codH{p}_{\mm, \sfd}(A).\\
         &C_0^{-p}\cdot \mm^{+,p}_{\mm, \sfd}(A) \le \mm^{+,p}_{\mm, \sfd'}(A) \le C_0^p\cdot\mm^{+,p}_{\mm,\sfd}(A).\\
         &C_0^{-p}\cdot {\rm cap}^{p}_{\mm,\sfd}(A) \le {\rm cap}^{p}_{\mm, \sfd'}(A) \le C_0^p\cdot{\rm cap}^{p}_{\mm,\sfd}(A).\\
         &C_0^{-p}\cdot {\rm AM}_{p,\sfd}(A,\mm) \le {\rm AM}_{p,\sfd'}(A,\mm) \le C_0^p\cdot {\rm AM}_{p,\sfd}(A,\mm),\\
        \end{aligned}
    \end{equation*}
\end{itemize}
The notations used here are self explaining.
\end{remark}

\section{Relations between the energies under Poincaré inequality}
\label{sec:relations_under_PI}
There are more connections between the several energies we introduced in the last section when the metric measure space $(\X,\sfd,\mm)$ is doubling and satisfies a $1$-Poincaré inequality. The first one is the representation formula for the perimeter, as proven in \cite[Theorem 4.2]{Amb02} and \cite[Theorem 4.6]{AmbrosioMirandaPallara2004}. The proposition, which is stated in a local version, is already present in the literature, but we are not able to find a proof. We include it for completeness.
\begin{proposition}
\label{prop:representation_perimeter_ambrosio}
    Let $(\X,\sfd,\mm)$ be $1$-{\rm PI} space. Let $U \subseteq \X$ be open and $A$ be a Borel set such that ${\rm Per}_\mm(A,U) < \infty$. Then
    \begin{equation}
    \label{eq:perimeter_representation}
        {\rm Per}_\mm(A,\cdot)=\theta \,\codH{1}_\mm\restr{\partial^e A \cap U},
    \end{equation}
    as measures on $U$, where $\theta\colon U\to [c_1,C_D]$ and $c_1 = c_1(C_D,C_P,\lambda) > 0$.
\end{proposition}

\begin{proof}
    Let us consider $f(x):=\sfd(x,U^c)$, that is $1$-Lipschitz. Applying the coarea inequality \eqref{eq:coarea_inequality_BV} we have that $\{ f > t\}$ has finite perimeter for a.e.\ $t \in \mathbb{R}$. We take a sequence $t_n \to 0$ for which $U_n :=\{ f > t_n\}$ has finite perimeter. By Remark \ref{rem:properties_perimeter}.(iii) we have 
    \begin{equation*}
        {\rm Per}_{\mm}(U_n \cap A,U) \le {\rm Per}_{\mm}(U_n,U) + {\rm Per}_{\mm}(A,U), 
    \end{equation*}
    thus ${\rm Per}_{\mm}(U_n \cap A,U)< \infty$. Indeed the measure ${\rm Per}_{\mm}(U_n,\cdot)$ is supported on $\partial U_n \subset U$ by Remark \ref{rem:properties_perimeter}.(v), so ${\rm Per}_{\mm}(U_n,U) = {\rm Per}_{\mm}(U_n,\X) < \infty$. For the same reason we have ${\rm Per}_{\mm}(U_n \cap A,\X)={\rm Per}_{\mm}(U_n \cap A,U) < \infty$. We are in position to apply \cite[Theorem 4.2]{Amb02}. So there exist a constant $c_1 = c_1(C_D,C_P,\lambda)$ and Borel functions $\theta_n \colon \X \to [c_1,C_D]$ such that 
    ${\rm Per}_\mm(A\cap U_n,\cdot)=\theta_n \codH{1}_\mm\restr{\partial^e (A \cap U_n)}$. We denote $\codH{1}_\mm\restr{\partial^e (A \cap U_n)}$ by $\mu_n$.\\
    {\bf Step 1:} $\mu_n(B) = \mu_m(B)$ if $m\geq n$ and $B \subseteq U_n$. \\
    This follows from the equalities
    \begin{equation}
        \label{eq:mu_n=mu_m}
        \mu_n (B) = \codH{1}_\mm(\partial^e A \cap B) = \mu_m(B)
    \end{equation}
    that we now justify. By \cite[Proposition 1.16]{BonPasqRajala2020}, where the finite perimeter assumption is not actually used, we get $\partial^e(A\cap U_n) \subseteq \partial^e A \cup \partial^e U_n$ and similarly for $m$. Therefore
    $$\mu_n(B) \leq \codH{1}_\mm(\partial^e A \cap B) + \codH{1}_\mm(\partial^e U_n \cap B) = \codH{1}_\mm(\partial^e A \cap B),$$
    because $B\cap \partial U_n = \emptyset$, and similarly for $m$. On the other hand we have
    $$\partial^e A = (\partial^e A \cap U_n) \cup (\partial^e A \cap U_n^c) \subseteq \partial^e (A \cap U_n) \cup (\partial^e A \cap U_n^c),$$
    where the last containment follows from the fact that if $x\in U_n$ then 
    $$\lims_{r\to 0} \frac{\mm(B_r(x) \cap A)}{\mm(B_r(x))} = \lims_{r\to 0} \frac{\mm(B_r(x) \cap (A \cap U_n))}{\mm(B_r(x))},\, \lims_{r\to 0} \frac{\mm(B_r(x) \setminus A)}{\mm(B_r(x))} = \lims_{r\to 0} \frac{\mm(B_r(x) \setminus (A \cap U_n))}{\mm(B_r(x))}$$
    because $U_n$ is open, and similarly for $m$. Thus
    $$\codH{1}_\mm(\partial^eA \cap B) \leq \codH{1}_\mm(\partial^e (A \cap U_n) \cap B) + \codH{1}_\mm(\partial^e A \cap U_n^c \cap B) = \mu_n(B)$$
    and similarly for $m$.\\
    {\bf Step 2:} for every $n$ there exists $V_n \subseteq U_n$ such that $\mu_m(V_n) = 0$ and every point of $U_n\setminus V_n$ is a Lebesgue point of $\theta_m$ with respect to the measure $\mu_m$, for every $m> n$.\\
    By \cite[Corollary 4.5]{Amb02}
    ${\rm Per}_\mm(A \cap U_m,\cdot)$ is asymptotically doubling and therefore $\mu_m$ is asymptotically doubling as well. This implies that the metric measure space $(\X,\sfd,\mu_m)$ is Vitali (cp. \cite[Theorem 3.4.3]{HKST15}). 
    By \cite[§3.4]{HKST15} the Lebesgue differentiation theorem holds on $(\X,\sfd,\mu_m)$. We apply it to $U_n$: we can find a Borel set $N_m \subseteq U_n$ with $\mu_m(N_m)=0$ and such that every point of $U_n \setminus N_m$ is a Lebesgue point of $\theta_m$ with respect to the measure $\mu_m$. We define $V_n := \bigcup_{m>n} N_m$ and we compute, for $m>n$,
    $$\mu_m(V_n) \leq \sum_{\ell > n} \mu_m(N_\ell) = \sum_{\ell > n} \mu_\ell(N_\ell) = 0,$$
    where for the first equality we applied Step 1.\\
    {\bf Step 3:} $\theta_m(x) = \theta_\ell(x)$ for every $x\in U_n\setminus V_n$ and every $m,\ell > n$.\\
    We apply Step 2 to get
    \begin{equation*}
        \begin{aligned}
            \theta_m(x) = \lim_{r \to 0} \dashint_{B_r(x)} \theta_m \,\d \mu_m &= \lim_{r \to 0} \frac{1}{\mu_m(B_r(x))} \int_{B_r(x)} \theta_m \,\d \mu_m \\
            &= \lim_{r \to 0} \frac{1}{\mu_m(B_r(x))}{\rm Per}_\mm(A\cap U_m, B_r(x)) \\
            &= \lim_{r \to 0} \frac{1}{\mu_m(B_r(x))}{\rm Per}_\mm(A,B_r(x))
        \end{aligned}
    \end{equation*}
    where we used Remark \ref{rem:properties_perimeter}.(i) and the fact that $B_r(x) \subseteq U_m$ if $r$ is small enough. The same computation holds for $\ell$. Moreover $B_r(x)$ is also included in $U_n$ if $r$ is small enough, so $\mu_m(B_r(x)) = \mu_\ell(B_r(x))$ if $r$ is small enough and $m,\ell > n$, by Step 1. This implies that $\theta_m(x) = \theta_\ell(x)$. 
    \vspace{1mm}

    \noindent Let us fix some additional notation. We set $\mu := \codH{1}_\mm\restr{(\partial^e A) \cap U}$ and $V:= \bigcup_n V_n$. For $x\in U \setminus V$ we set $\theta(x) := \theta_m(x)$, where $m > n$ and $x\in U_n$.\\
    {\bf Step 4:} $\mu(V) = 0$ and $\theta$ is well defined on $U\setminus V$, defining a Borel function on $U$ with values on $[c_1,C_D]$.\\
    We have
    $$\mu(V) = \codH{1}_\mm \left( \bigcup_n \partial^e A \cap V_n\right) \leq \sum_n \codH{1}_\mm( \partial^e A \cap V_n) = \sum_n \mu_{n+1}(V_n) = 0,$$
    where in the second to last equality we used \eqref{eq:mu_n=mu_m}. We need to show that $\theta$ is well defined. We fix $x\in U \setminus V$. By definition $x$ belongs to $U_n \setminus V_n$ for some $n$. Step 3 says that the quantity $\theta_m(x)$ does not depend on the choice of $m > n$. Moreover suppose that $x$ belongs to $U_n \setminus V_n$ and $U_m \setminus V_m$. Choosing $\ell > m,n$ we get that the assignment $\theta(x) = \theta_\ell(x)$ is well defined and does not depend on $n$ and $m$, so $\theta$ is well defined on $U\setminus V$. 
    By definition $c_1\leq \theta \leq C_D$ because each $\theta_n$ has this property. Finally it is not difficult to prove that $\theta$ is Borel using that each $\theta_n$ is Borel.\\
    {\bf Step 5:} 
    ${\rm Per}_\mm(A\cap U, B) = \theta\,\codH{1}_\mm\restr{\partial^e A \cap U}(B)$ for every $B\subseteq U$.\\
    First of all we prove the statement for an open set $B$ such that $\sfd(B,U^c) > 0$, so that $B\subseteq U_n$ if $n$ is big enough. Using that $\theta(x) = \theta_n(x)$ for $\mu$ and $\mu_n$ almost every $x\in B$ and that $\mu = \mu_n$ on $B$ by \eqref{eq:mu_n=mu_m} we get
    $$(\theta\mu)(B) = \int_{B}\theta\,\d\mu = \int_{B}\theta_n\,\d\mu_n = {\rm Per}_\mm(A\cap U_n, B) = {\rm Per}_\mm(A, B).$$
    The last equality follows by Remark \ref{rem:properties_perimeter}.(i). For a generic open $B\subseteq U$ we define $B_n = B\cap U_n$ and we observe that $B$ is the increasing union of the sets $B_n$. Since both $\mu$ and ${\rm Per}_\mm(A,\cdot)$ are Borel measures we get
    $$(\theta\mu)(B) = \lim_{n\to +\infty} (\theta\mu)(B_n) = \lim_{n\to +\infty}{\rm Per}_\mm(A, B_n) = {\rm Per}_\mm(A, B).$$
    An application of monotone class theorem gives that $(\theta\mu)(B) ={\rm Per}_\mm(A, B)$ for every Borel set $B \subseteq U$.
\end{proof}

Given a set, it is possible to check the fact that it has finite perimeter by checking the finiteness of the codimension-$1$ Hausdorff measure of the essential boundary: this is the content of \cite[Theorem 1.1]{Lahti2020}. We extend the characterization by using also the approximate modulus. We specialize also to the case of separating sets.

\begin{proposition}
\label{prop:finite_perimeter_equivalences}
    Let $(\X,\sfd,\mm)$ be a $1$-{\rm PI} space.
    Let $U\subseteq \X$ be an open set and $A \subseteq \X$ be a Borel set. Then the following are equivalent:
    \begin{itemize}
        \item[(i)] ${\rm Per}_\mm(A,U) < \infty$;
        \item[(ii)] $\codH{1}_\mm(\partial^e A \cap U)< \infty$;
        \item[(iii)] ${\rm AM}_1(\partial^e A \cap U, \mm) < \infty$.
    \end{itemize}
     If $x,y\in \X$, $L\geq 1$, $\min\lbrace \sfd(x,\partial^e A), \sfd(y,\partial^e A)\rbrace > 0$, $U\subseteq \overline{B}_{x,y}^L$ then (i)-(iii) are equivalent to
     \begin{itemize}
         \item[(iv)] $\codH{1}_{\mm_{x,y}^L}(\partial^e A \cap U)< \infty$;
         \item[(v)] ${\rm AM}_1(\partial^e A \cap U, \mm_{x,y}^L) < \infty$.
     \end{itemize}
\end{proposition}
\begin{proof}
    The implication from (i) to (ii) follows directly from Proposition \ref{prop:representation_perimeter_ambrosio} while (ii) implies (i) is \cite[Theorem 1.1]{Lahti2020}. The equivalence between (ii) and (iii) is a direct consequence of \cite[Theorem 4.4]{HMM2022}, where it is shown that $\codH{1}_\mm(\cdot)$ and ${\rm AM}_1(\cdot, \mm)$ are comparable for Borel sets if $(\X,\sfd,\mm)$ is a $1$-PI space. \\
    We now move to the second part.
    If $A$ is as in the assumptions then there exists $M>0$ such that $1/M \leq R_{x,y}^L \leq M$ on a small neighbourhood of $\partial^e A$. 
    This implies by a direct estimate that $\codH{1}_\mm(\partial^e A \cap U)< \infty$ if and only if $\codH{1}_{\mm_{x,y}^L}(\partial^e A \cap U)< \infty$. Moreover ${\rm AM}_1(\partial^e A \cap U, \mm) < \infty$ if and only if ${\rm AM}_1(\partial^e A \cap U, \mm_{x,y}^L) < \infty$ because, thanks to Lemma \ref{lemma:AM_localized_near_boundary}, both the approximate moduli can be computed with curves belonging to the above mentioned tubular neighbourhood of $\partial^e A$.
\end{proof}

The second relation we explore is between the approximate modulus and the codimension Hausdorff measure computed with respect to $\mm_{x,y}^L$. First we state it in case of length spaces, later we will see how to get the statement under the Poincaré inequality assumption. 
\begin{proposition}
    \label{prop:am<Hcod_length_spaces}
    Let $(\X,\sfd,\mm)$ be a doubling metric measure space such that $(\X,\sfd)$ is geodesic. Then for every $x,y\in \X$ and every Borel set $A\subseteq \X$ with $\min\lbrace\sfd(x,A), \sfd(y,A)\rbrace >0$ it holds 
    $${\rm AM}_p(A,\mm_{x,y}^L) \le C_D \cdot \codH{p}_{\mm_{x,y}^{2L}}(A).$$
\end{proposition}

The same statement is proved in \cite[Theorem 2.1]{HMM2022} for the same quantities \emph{computed with respect to} $\mm$. We adapt their proof to the measure $\mm_{x,y}^L$ by localizing the argument. In order to get a nice local estimate we use the assumption that $(\X,\sfd)$ is a length space to have the continuity of the Riesz potential $R_{x,y}$.

\begin{proof}
    We fix $A\subseteq \X$ as in the assumptions. Let us consider $r \ll 1$ so that $\overline{B}_r(A) \subset \X \setminus \lbrace x,y \rbrace$. We show that 
    \begin{equation}
        \label{eq:am<codH_auxiliary}
        {\rm AM}_p(A,\mm_{x,y}^L) = {\rm AM}_p(A \cap \overline{B}_{x,y}^L,\mm_{x,y}^L) \leq {\rm AM}_p(A \cap \overline{B}_{x,y}^L,\mm_{x,y}^{2L}) \le C_D \cdot \codH{p}_{\mm_{x,y}^{2L}}(A \cap \overline{B}_{x,y}^L),
    \end{equation}
    which is enough for the thesis. The first equality is true because the approximate modulus is an outer measure and ${\rm AM}_p(A\cap (\overline{B}_{x,y}^L)^c,\mm_{x,y}^L) = 0$. This is true since    
    $(\overline{B}_{x,y}^L)^c$ is open and we can argue as in Lemma \ref{lemma:AM_localized_near_boundary} to see that ${\rm AM}_p(A\cap (\overline{B}_{x,y}^L)^c,\mm_{x,y}^L)$ can be computed on a family of curves which is entirely contained in $(\overline{B}_{x,y}^L)^c$. Therefore every admissible sequence can be taken to be zero on $\overline{B}_{x,y}^L$. The second inequality is trivial, since $\mm_{x,y}^L \le \mm_{x,y}^{2L}$. We focus on the third one.\\    
    The function $R_{x,y}$ is continuous at every point of $\X \setminus \lbrace x,y \rbrace$ by Remark \ref{lemma:continuity_riesz_potential}, so it is uniformly continuous on the compact set $\overline{B}_r(A) \cap \overline{B}_{x,y}^{2L}$. For every $\varepsilon>0$, there exists $\eta > 0$ such that if $z,z'\in \overline{B}_r(A) \cap \overline{B}_{x,y}^{2L}$ and $\sfd(z,z') <\eta$ then $|R_{x,y}(z) - R_{x,y}(z')| \le \varepsilon$. 
    We can always assume $\eta \le r$. Moreover $R_{x,y}$ is continuous and positive on $\overline{B}_r(A) \cap \overline{B}_{x,y}^{2L}$, so there exists a constant $M$ such that 
    \begin{equation}
        \label{eq:proof-AM<=cod-H}
        M^{-1}\le R_{x,y}\le M \text { on }\overline{B}_r(A) \cap \overline{B}_{x,y}^{2L}.
    \end{equation}
    Assume that $\codH{p}_{\mm_{x,y}^{2L}}(A \cap \overline{B}_{x,y}^L)<\infty$, otherwise there is nothing to prove. For $j \in \N$ such that $\frac{1}{j} < \frac{1}{2}\min\lbrace \eta, L\sfd(x,y)\rbrace$ we choose $\{ B_{r_{i}^j}(z_{i}^j) \}_i$ such that $A \cap \overline{B}_{x,y}^L \subset \bigcup_{i} B_{r_{i}^j}(z_{i}^j)$, $r_{i}^j<1/j$ and
\begin{equation}
    \sum_i \frac{\mm_{x,y}^{2L}(B_{r_{i}^j}(z_{i}^j))}{(r_{i}^j)^p}\le \codH{p}_{\mm_{x,y}^{2L}}(A \cap \overline{B}_{x,y}^L)+\frac{1}{j}.
\end{equation}
Without loss of generality we can suppose that $B_{r_{i}^j}(z_{i}^j) \cap (A \cap \overline{B}_{x,y}^L) \neq \emptyset$, so $B_{2r_{i}^j}(z_{i}^j) \subseteq \overline{B}_{x,y}^{2L}$.
We define 
\begin{equation}
    \rho_j:=\left( \sum_i \frac{\chi_{B_{2r_{i}^j}(z_{i}^j)}}{(r_{i}^j)^p} \right)^{\frac{1}{p}}
\end{equation}
and we follow verbatim the proof in \cite[Thm.\ 2.1]{HMM2022} to show that $\{\rho_j\}$ is admissible, i.e.\ $\limi_{j\to \infty}\int_\gamma \rho_j \,\d s \ge 1$ for every $\gamma \in \Gamma(A \cap \overline{B}_{x,y}^L)$.
Thus, we have
\begin{equation}
\label{eq:am<Hcod_2}
    {\rm AM}_p(\Gamma(A \cap \overline{B}_{x,y}^L),\mm_{x,y}^{2L}) \le \limi_{j \to \infty} \int \rho_j^p\,\d \mm_{x,y}^{2L} \le \limi_{j \to \infty} \sum_{i} \frac{\mm_{x,y}^{2L}(B_{2 r_{i}^j}(z_{i}^j))}{(r_{i}^j)^p}. 
\end{equation}

\noindent By uniform continuity and \eqref{eq:proof-AM<=cod-H} we have 
    \begin{equation}
    \begin{aligned}
        \mm_{x,y}^{2L}(B_{2r_{i}^j}(z_{i}^j)) =  \int_{B_{2r_{i}^j}(z_{i}^j)} R_{x,y}\,\d\mm &\leq (R_{x,y}(z_{i}^j) + \varepsilon) \mm(B_{2r_{i}^j}(z_{i}^j)) \\
        &\leq R_{x,y}(z_{i}^j)(1+\varepsilon M)\mm(B_{2r_{i}^j}(z_{i}^j))\\
        &\leq C_D\cdot R_{x,y}(z_{i}^j)(1+\varepsilon M)\mm(B_{r_{i}^j}(z_{i}^j))
    \end{aligned}
    \end{equation}
    and similarly 
    \begin{equation}
    \begin{aligned}
        \mm_{x,y}^{2L}(B_{r_{i}^j}(z_{i}^j)) =  \int_{B_{r_{i}^j}(z_{i}^j)} R_{x,y}\,\d\mm &\geq R_{x,y}(z_{i}^j)(1-\varepsilon M )\mm(B_{r_{i}^j}(z_{i}^j)).
    \end{aligned}
    \end{equation}
Therefore 
\begin{equation}
    \mm_{x,y}^{2L}(B_{2r_{i}^j}(z_{i}^j)) \leq C_D \frac{1+\varepsilon M}{1- \varepsilon M} \mm_{x,y}^{2L}(B_{r_{i}^j}(z_{i}^j)).
\end{equation}

\noindent This, together with \eqref{eq:am<Hcod_2}, gives

\begin{equation}
\begin{aligned}
    {\rm AM}_p(A \cap \overline{B}_{x,y}^L,\mm_{x,y}^{2L}) %&\le C_D \frac{1+\varepsilon M}{1- \varepsilon M}\, \limi_{j \to \infty} \,\left(\codH{p}_{\mm_{x,y}^{L+\tau}}(A \cap \overline{B}_{x,y}^L)+\frac{1}{j}\right)\\
    \le C_D \frac{1+\varepsilon M}{1- \varepsilon M}\codH{p}_{\mm_{x,y}^{2L}}(A \cap \overline{B}_{x,y}^L).
\end{aligned}
\end{equation}
Taking the limit for $\varepsilon$ going to zero we end the proof of \eqref{eq:am<codH_auxiliary}. 
\end{proof}
We now state the inequality between approximate modulus and codimension Hausdorff measure for PI spaces.
\begin{corollary}
    \label{cor:AM<Hcod_Poincaré}
    Let $(\X,\sfd,\mm)$ be a $p$-{\rm PI} space. Then for every $x,y\in \X$ and every Borel set $A\subseteq \X$ with $\min\lbrace\sfd(x,A), \sfd(y,A)\rbrace >0$ it holds 
    $${\rm AM}_p(A,\mm_{x,y}^L) \le C \cdot \codH{p}_{\mm_{x,y}^{2L}}(A),$$
    where $C = C(C_D,C_P,\lambda)$.
\end{corollary}
\begin{proof}
    By Remark \ref{rem:bilipschitz_equivalence} there exists another metric $\sfd'$ on $\X$ which is geodesic and $C_0$-equivalent to $\sfd$, with $C_0$ depending only on the structural constants.  Proposition \ref{prop:am<Hcod_length_spaces} implies
    $${\rm AM}_{p,\sfd'}(A,\mm_{x,y}^L) \le C'\cdot \codH{p}_{\mm_{x,y}^{2L},\sfd'}(A),$$
    for a constant $C'$ depending only on the structural constants, since $(\X,\sfd',\mm)$ is doubling by Remark \ref{rem:biLipschitz_change}. The same remark gives the thesis.
\end{proof}

\section{Relations between the different energies weighted with the Riesz potential}
\label{sec:relations_weighted_with_riesz}
Among the energies we have introduced there are two of them that are measures: the perimeter and the codimension Hausdorff measure. In Theorem \ref{theo:main-intro-p=1} we will consider their weighted versions, where the weight is the Riesz potential $R_{x,y}^L$. There are actually two ways to do so: either to consider the measures ${\rm Per}_{\mm_{x,y}^L}$ and $\codH{p}_{\mm_{x,y}^L}$ or to take the measures $R_{x,y}^L{\rm Per}_\mm$ and $R_{x,y}^L\codH{p}_\mm$. In this section we show that the two approaches are comparable if the space satisfies a Poincaré inequality. We start by studying the case of the perimeter. We focus directly to the case of the separating sets.

\begin{lemma}
\label{lemma:comparison_perimeter_mmxy}
Let $(\X,\sfd,\mm)$ be a metric measure space. Let $x,y\in \X$ and let $\Omega \in \SS_{\textup{top}}(x,y)$. Assume that $R_{x,y}^L$ is continuous on $B_{x,y}^L \setminus \lbrace x,y\rbrace$. Then
    \begin{equation}
    \label{eq:inequality_weighted_perimeter_open}
    \frac{1}{2}\,{\rm Per}_{\mm_{x,y}^L}(\Omega,A) \le \int_A R_{x,y}^L\,\d {\rm Per}_\mm(\Omega,\cdot) \le 2\,{\rm Per}_{\mm_{x,y}^L}(\Omega,A)
    \end{equation}
    for every $A \subseteq \X$ open.
Moreover, if ${\rm Per}_\mm(\Omega, B_{x,y}^L) < \infty$, we have that for every Borel set $B\subseteq \X$
    \begin{equation}
    \label{eq:inequality_weighted_perimeter_as_measures}
        \frac{1}{2}\,{\rm Per}_{\mm_{x,y}^L}(\Omega,B) \le \int_B R_{x,y}^L\,\d{\rm Per}_\mm(\Omega,\cdot) \le 2\,{\rm Per}_{\mm_{x,y}^L}(\Omega,B).
    \end{equation}
\end{lemma}
The assumption on the continuity of the Riesz potential is satisfied for instance if $(\X,\sfd,\mm)$ is doubling and $(\X,\sfd)$ is geodesic, see Lemma \ref{lemma:continuity_riesz_potential}.
\begin{proof}
    We fix $\varepsilon > 0$. We define
    \begin{equation*}
        A_i^{\varepsilon}:=\left\{ z \in B_{x,y}^L:\,\left(i-1-\frac{1}{4}\right)\varepsilon < R_{x,y}^L(z) < \left(i+\frac{1}{4}\right)\varepsilon  \right\}
    \end{equation*}
    for $i\in\N$. By definition we have  $A_i^\varepsilon\subseteq B_{x,y}^L$ and $\bigcup_i A_i^\varepsilon = B_{x,y}^L \setminus \lbrace x,y\rbrace$. The sets $A_i^\varepsilon$ are open because $R_{x,y}^L$ is continuous on $B_{x,y}^L$. Let $A\subseteq \X$ be open. Using directly the definition of perimeter for open sets, we get
    \begin{equation*}
        \left(i-1-\frac{1}{4}\right)\varepsilon\,{\rm Per}_{\mm}(E,A_i^\varepsilon \cap A)\le {\rm Per}_{\mm_{x,y}^L}(E,A_i^\varepsilon \cap A) \le \left(i+\frac{1}{4}\right)\varepsilon\,{\rm Per}_{\mm}(E,A_i^\varepsilon \cap A).
    \end{equation*}
    for every Borel set $E\subseteq \X$. Let $\Omega \in \SS_{\text{top}}(x,y)$. Its topological boundary, which contains the support of ${\rm Per}_\mm(\Omega, \cdot)$ by Remark \ref{rem:properties_perimeter}.(v), has positive distance from $x$ and $y$. Then we compute
    \begin{equation*}
    \begin{aligned}
        \int_{A} R_{x,y}^L\,\d {\rm Per}_\mm(\Omega,\cdot)&=\int_{(B_{x,y}^L \setminus \lbrace x,y\rbrace)\cap A} R_{x,y}^L\,\d {\rm Per}_\mm(\Omega,\cdot) \\
        &\le \sum_{i=0}^\infty \int_{A_i^{\varepsilon} \cap A} R_{x,y}^L\,\d {\rm Per}_\mm(\Omega,\cdot) \\
        &\le \sum_{i=0}^\infty \left(i +\frac{1}{4}\right)\varepsilon {\rm Per}_\mm(\Omega,A_i^{\varepsilon} \cap A)\\
        & \le \sum_{i=0}^\infty {\rm Per}_{\mm_{x,y}^L}(\Omega, A^\varepsilon_i \cap A) + \sum_{i=0}^\infty \frac{3}{2}\varepsilon {\rm Per}_\mm(\Omega, A^\varepsilon_i \cap A) \\
        &\le 2{\rm Per}_{\mm_{x,y}^L}(\Omega,A)+ 3\varepsilon {\rm Per}_\mm(\Omega, B_{x,y}^L \cap A).\\
    \end{aligned}
    \end{equation*}
In the last inequality, we used that $\sum_{i \text{ even}} {\rm Per}_{\mm_{x,y}^L}(\Omega, A^\varepsilon_i \cap A) \le {\rm Per}_{\mm_{x,y}^L}(\Omega,A)$, since $A_i^\varepsilon \cap A_{i+2}^\varepsilon = \emptyset$ for every $i \in \mathbb{N}$ and similarly for the set of odd indexes. The same properties have been used for the estimate of the additional term ${\rm Per}_\mm(\Omega, \cdot)$ in the right hand side. Now we distinguish two cases. If ${\rm Per}_\mm(\Omega, A\cap B_{x,y}^L)=\infty$, then also ${\rm Per}_\mm(\Omega, A)=\infty$  and inequality \eqref{eq:inequality_weighted_perimeter_open} is trivially satisfied. If ${\rm Per}_\mm(\Omega, A\cap B_{x,y}^L)<\infty$ we take the limit as $\varepsilon$ converging $0$ in the previous inequality. In any case
$$\int_{A} R_{x,y}^L\,\d {\rm Per}_\mm(\Omega,\cdot) \leq 2{\rm Per}_{\mm_{x,y}^L}(\Omega,A).$$
In a similar way we get
$$ {\rm Per}_{\mm_{x,y}^L}(\Omega,A)\leq 2\int_A R_{x,y}^L\,\d{\rm Per}_\mm(\Omega,\cdot) + 3\varepsilon{\rm Per}_\mm(\Omega, A\cap B_{x,y}^L).$$
If ${\rm Per}_\mm(\Omega, A\cap B_{x,y}^L) = \infty$ then also $\int_A R_{x,y}^L\,\d{\rm Per}_\mm(\Omega,\cdot) = \infty$ because there exists $m>0$ such that $R_{x,y}^L \geq m$ on $A\cap B_{x,y}^L$. Otherwise, if ${\rm Per}_\mm(\Omega, A\cap B_{x,y}^L) < \infty$, we can let $\varepsilon$ converge to $0$ and obtain
$$ {\rm Per}_{\mm_{x,y}^L}(\Omega,A)\leq 2\int_A R_{x,y}^L\,\d{\rm Per}_\mm(\Omega,\cdot),$$
thus concluding the proof of \eqref{eq:inequality_weighted_perimeter_open}.
The proof of \eqref{eq:inequality_weighted_perimeter_as_measures} follows by the previous step and an application of monotone class theorem.
\end{proof}

\begin{corollary}
    \label{cor:perimeter_comparison_Poincaré}
    Let $(\X,\sfd,\mm)$ be a $p$-{\rm PI} space. Let $x,y\in \X$ and let $\Omega \in \SS_{\textup{top}}(x,y)$. Then
    \begin{equation*}
    \frac{1}{C}\,{\rm Per}_{\mm_{x,y}^L}(\Omega,\X) \le \int_\X R_{x,y}^L\,\d {\rm Per}_\mm(\Omega,\cdot) \le C\,{\rm Per}_{\mm_{x,y}^L}(\Omega,\X),
    \end{equation*}
    where $C = C(C_D,C_P,\lambda)$.
\end{corollary}
\begin{proof}
    Using Remark \ref{rem:bilipschitz_equivalence} we find a new metric $\sfd'$ on $\X$ which is geodesic and $C_0$-equivalent to $\sfd$, with $C_0$ depending only on the structural constants.  Lemma \ref{lemma:comparison_perimeter_mmxy} can be applied to the metric measure space $(\X,\sfd',\mm)$, because of Lemma \ref{lemma:continuity_riesz_potential}. We conclude by applying Remark \ref{rem:biLipschitz_change}.
\end{proof}

We now move to a similar comparison between the two possible versions of the codimension Hausdorff measures. Let us fix some notation. Let $f\colon \X \times [0,+\infty) \to [0,+\infty)$ be a gauge function. For $A\subseteq \X$ and $\delta > 0$ we set 
$$\mathcal{H}_\delta^f(A) := \inf\left\lbrace \sum_{i=1}^\infty f(z_i,r_i) \text{ s.t. } A \subseteq \bigcup_i B_{r_i}(z_i),\, r_i \leq \delta \right\rbrace$$
and
$$\mathcal{H}^f(A) := \lim_{\delta \to 0} \mathcal{H}_\delta^f(A).$$
\begin{lemma}
    $\mathcal{H}^f$ is a measure on the Borel sets.
\end{lemma}
\begin{proof}
    From the definition it is an outer measure on $\X$. Moreover if $A,B \subseteq \X$ are separated, i.e. $\sfd(A,B) > 0$, then $\mathcal{H}^f(A \cup B) = \mathcal{H}^f(A) + \mathcal{H}^f(B)$. By \cite[§1.1, Theorem 5]{Evans} every Borel set is measurable.
\end{proof}

Let $p\geq 1$, $x,y \in \X$ and $L\geq 1$. We consider two gauge functions:
$$f(z,r):= \frac{\mm_{x,y}^L(B_r(z))}{r^p}, \qquad \qquad g(z,r) := R_{x,y}^L(z)\cdot \frac{\mm(B_r(z))}{r^p}.$$

\begin{lemma}
\label{lemma:comparison_hausdorff_1}
    Let $(\X,\sfd,\mm)$ be a proper metric measure space. Let $x,y \in \X$, $L,p\geq 1$. Let $A \subseteq \X$ such that $\min \lbrace \sfd(x,A), \sfd(y,A)\rbrace > 0$. Assume $R_{x,y}$ is continuous on $\overline{B}_{x,y}^L \setminus \lbrace x,y\rbrace$. Then
    \begin{equation}
        \codH{p}_{\mm_{x,y}^L}(A) = \mathcal{H}^{f}(A) = \mathcal{H}^{g}(A).
    \end{equation}
\end{lemma}

\begin{proof}
    The first equality is the definition of $\codH{p}_{\mm_{x,y}^L}(A)$. We focus on the other equality. We fix $A\subseteq \X$ as in the assumptions. Let us consider $r \ll 1$ so that $\overline{B}_r(A) \subset \X \setminus \lbrace x,y \rbrace$. By assumption the function $R_{x,y}$ is continuous at every point of $\overline{B}_{x,y}^L \setminus \lbrace x,y \rbrace$, so it is uniformly continuous on the compact set $\overline{B}_r(A) \cap \overline{B}_{x,y}^L$. For every $\varepsilon> 0$, there exists $\eta > 0$ such that if $z,z'\in \overline{B}_r(A) \cap \overline{B}_{x,y}^L$ and $\sfd(z,z') <\eta$ then $|R_{x,y}(z) - R_{x,y}(z')| \le \varepsilon$. 
    We can always assume $\eta \le r$. Moreover $R_{x,y}$ is continuous and positive on $\overline{B}_r(A) \cap \overline{B}_{x,y}^L$, so there exists a constant $M$ such that 
    \begin{equation}
        \label{eq:boundedness_riesz_intheproofofcomparisonH}
        M^{-1}\le R_{x,y}\le M \text { on }\overline{B}_r(A) \cap \overline{B}_{x,y}^L.
    \end{equation}
    We prove first that $\mathcal{H}^{f}(A) \le \mathcal{H}^{g}(A)$. Fix $\varepsilon > 0$ and find the corresponding $\eta > 0$ as above. Fix also an arbitrary $\tau > 0$ and consider the set $A_\tau := A\cap \overline{B}_{x,y}^{L-\tau}$.
    We fix $\delta \le \min\lbrace\eta, \frac{\tau \sfd(x,y)}{2}\rbrace$ and we take a collection $\{ B_{r_i}(z_i)\}_{i=1}^{\infty}$ such that $A_\tau \subset \cup_{i=1}^\infty B_{r_i}(z_i)$ and $r_i \le \delta$. We can suppose that $\sfd(z_i,A_\tau) \leq \frac{\tau \sfd(x,y)}{2}$, otherwise $B_{r_i}(z_i) \cap A_\tau = \emptyset$ and we can remove $B_{r_i}(z_i)$ from the covering. Therefore $B_{r_i}(z_i) \subseteq B_{x,y}^L$ and $R_{x,y}^L = R_{x,y}$ on $B_{x,y}^L$.    
    By uniform continuity we have
    \begin{equation}
        \left| \int_{B_{r_i}(z_i)} R_{x,y}^L\,\d \mm - R_{x,y}^L(z_i) \mm(B_{r_i}(z_i))\right| \le \varepsilon \mm(B_{r_i}(z_i))
    \end{equation}
    Thus, summing up, we get
    \begin{equation}
    \label{eq:sum_estimate_riesz_hausdorff}
        \left| \sum_{i=1}^{\infty}\frac{\int_{B_{r_i}(z_i)} R_{x,y}^L\,\d \mm}{r_i^p} - \sum_{i=1}^{\infty}\frac{R_{x,y}^L(z_i) \mm(B_{r_i}(z_i))}{r_i^p} \right| \le \varepsilon \sum_{i=1}^{\infty}\frac{\mm(B_{r_i}(z_i))}{r_i^p}.
    \end{equation}
    From \eqref{eq:sum_estimate_riesz_hausdorff} and the fact that  $1 \le M R_{x,y}^L(z_i)$ for every $i$, we get
    \begin{equation}
        \sum_{i=1}^{\infty}\frac{\int_{B_{r_i}(z_i)} R_{x,y}^L\,\d \mm}{r_i^p} \le \left( 1+M\varepsilon \right)\,\sum_{i=1}^{\infty}\frac{R_{x,y}^L(z_i) \mm(B_{r_i}(z_i))}{r_i^p}. 
    \end{equation}
    Taking the infimum over all possible coverings with balls of radius smaller than $\delta$, this gives $\mathcal{H}_{\delta}^{f}(A_\tau) \le (1+M\varepsilon)\,\mathcal{H}_{\delta}^{g}(A_\tau)$. Taking the limit as $\delta \to 0$, we get that $\mathcal{H}^{f}(A_\tau) \le (1+M\varepsilon)\,\mathcal{H}^{g}(A_\tau)$ and then, as $\varepsilon$ goes to zero, we have $\mathcal{H}^{f}(A_\tau) \le \mathcal{H}^{g}(A_\tau)$. Now we can pass to the limit for $\tau$ going to zero to conclude that $\mathcal{H}^{f}(A) \le \mathcal{H}^{g}(A)$, because both $\mathcal{H}^{f}$ and $\mathcal{H}^{g}$ are measures.
    The converse inequality $\mathcal{H}^{g}(A) \le \mathcal{H}^{f}(A)$ follows by considering the other inequality given by \eqref{eq:sum_estimate_riesz_hausdorff} and using a similar argument. \\
    %Finally, using again the uniform continuity of $R_{x,y}^L$ on a neighbourhood of $A$ one can prove also the last equality, namely $\mathcal{H}^g(A) = R_{x,y}^L\codH{p}_\mm(A)$. We omit the details.
\end{proof}

The next step is to link $\mathcal{H}^g$ to $R_{x,y}^L\codH{p}_\mm$.

\begin{lemma}
    \label{lemma:comparison_hausdorff_2}
    Let $(\X,\sfd,\mm)$ be a proper metric measure space. Let $x,y\in \X$, $L,p\geq 1$. Let $A\subseteq \X$ such that $\min\lbrace \sfd(x,A), \sfd(y,A)\rbrace > 0$.
    Assume $R_{x,y}$ is continuous on $\overline{B}_{x,y}^L\setminus \lbrace x,y \rbrace$. Then
    \begin{equation}
    \label{eq:comparison_codH_with_Riesz_in_the_middle}
    \frac{1}{2} \int_{A} R_{x,y}^L \,\d \codH{p}_\mm \le       \mathcal{H}^g(A) \le 2 \int_A R_{x,y}^L \,\d \codH{p}_\mm.
    \end{equation}
\end{lemma}

\begin{proof}
    Let us start with the second inequality.
    We fix $A\subseteq \X$ as in the assumptions. Let us consider $r \ll 1$ so that $\overline{B}_r(A) \subset \X \setminus \lbrace x,y \rbrace$. As recalled above the function $R_{x,y}$ is continuous at every point of $\overline{B}_{x,y}^L \setminus \lbrace x,y \rbrace$, so it is uniformly continuous on the compact set $\overline{B}_r(A) \cap \overline{B}_{x,y}^L$. We fix $\varepsilon, \varepsilon' > 0$. Then there exists $\eta > 0$ such that if $\sfd(z,z')\le \eta$ with $z,z' \in \overline{B}_r(A) \cap \overline{B}_{x,y}^L$ then $|R_{x,y}(z)-R_{x,y}(z')|\le \varepsilon'$. 
    We define $A_\tau:=A \cap \overline{B}_{x,y}^{L-\tau}$.
    Let us define $E_j^\varepsilon:=\{ z \in B_{x,y}^L : (j-1 - \frac{1}{4})\varepsilon < R_{x,y}^L(z) < (j + \frac{1}{4}) \varepsilon\}$. Notice that $\bigcup_{j} E_j^\varepsilon = B_{x,y}^L \setminus \lbrace x,y\rbrace$.\\
    Let us fix $0<\delta \leq \min\lbrace r,\eta,\frac{\tau\sfd(x,y)}{2}\rbrace$. 
    We define $\varepsilon_j:=(\left(j + \frac{1}{4}\right) \varepsilon +\varepsilon')^{-1}\,2^{-j}\varepsilon$. Let us consider a collection $\{ B_{r_i^j}(z_i^j)\}_i$ with $r^i_j\le \delta$ for every $i$ and such that
    \begin{equation}
        E_j^\varepsilon \cap A_\tau \subseteq \bigcup_{i=1}^\infty B_{r_i^j}(z_i^j),\qquad \sum_{j} \frac{\mm(B_{r_i^j}(z_i^j))}{(r_i^j)^p} \le \codH{p}_{\mm,\delta}(E_j^\varepsilon \cap A_\tau) + \varepsilon_j.
    \end{equation}
    Notice that $A_{\tau} \subseteq \bigcup_{i,j} B_{r_i^j}(z_i^j)$, thus it is a competitor for the estimate of $\mathcal{H}^g_\delta(A_\tau)$.
    We compute
    {\allowdisplaybreaks
        \begin{align}
            \mathcal{H}^g_\delta(A_\tau) &\le \sum_{i,j}^\infty R_{x,y}^L(z_i^j) \frac{\mm(B_{r_i^j}(z_i^j))}{(r_i^j)^p} \le \sum_j \left(\left(j + \frac{1}{4}\right) \varepsilon +\varepsilon'\right)\sum_{i=1}^\infty \frac{\mm(B_{r_i^j}(z_i^j))}{(r_i^j)^p}\\
        & \le \sum_j \left(\left(j + \frac{1}{4}\right) \varepsilon +\varepsilon'\right)\left(\codH{p}_{\mm,\delta}(E_j^\varepsilon \cap A_\tau) + \varepsilon_j\right)\\
        & \le \sum_j \left(\left(j + \frac{1}{4}\right) \varepsilon +\varepsilon'\right)\codH{p}_{\mm}(E_j^\varepsilon \cap A_\tau) + \sum_j 2^{-j}\varepsilon \\
        &=\sum_j \int_{E_j^\varepsilon \cap A_\tau} \left(\left(j+\frac{1}{4}\right) \varepsilon +\varepsilon'\right)\,\d\codH{p}_{\mm} +\varepsilon \\
        & \le\sum_j \int_{E_j^\varepsilon \cap A_\tau} \left(R_{x,y}^L +\varepsilon'+
        \frac{3}{2}\varepsilon\right)\,\d\codH{p}_{\mm} +\varepsilon\\
        & \le\sum_{j \text{ even}} \int_{E_j^\varepsilon \cap A_\tau} \left(R_{x,y}^L +\varepsilon'+
        \frac{3}{2}\varepsilon\right)\,\d\codH{p}_{\mm} \\
        &+ \sum_{j \text{ odd}} \int_{E_j^\varepsilon \cap A_\tau} \left(R_{x,y}^L +\varepsilon'+
        \frac{3}{2}\varepsilon\right)\,\d\codH{p}_{\mm} +\varepsilon\\
        &\le 2 \int_{A_\tau} R_{x,y}^L \,\d \codH{p}_\mm + (2\varepsilon' + 3\varepsilon) \codH{p}_\mm(A_\tau) + \varepsilon.
        \end{align}
    }
    By taking the limit as $\delta \to 0$ and as $\tau \to 0$, we get
    \begin{equation}
        \mathcal{H}^g(A)\le 2 \int_{A} R_{x,y}^L \,\d \codH{p}_\mm + (2\varepsilon' + 3\varepsilon) \codH{p}_\mm(A) + \varepsilon.
    \end{equation}
    Now we distinguish two cases. If $\codH{p}_\mm(A) =\infty$, then also $\int_A R_{x,y}^L\,\d\codH{p}_\mm = \infty$ and there is nothing to prove.
    If $\codH{p}_\mm(A) < \infty$, we can take the limit as $\varepsilon$, $\varepsilon'$ converging to $0$, concluding the proof of the right inequality in \eqref{eq:comparison_codH_with_Riesz_in_the_middle}.\\
    Let us prove the first inequality. We fix $\varepsilon, \varepsilon' > 0$, $\eta > 0$ and $A_\tau$ as above. We notice that $A_\tau$ intersects a finite number of sets $E_j^\varepsilon$ because $R_{x,y}^L$ is bounded on $A_\tau$. Therefore we can find $\delta' > 0$ such that $\codH{p}_{\mm}(A_\tau \cap E_j^\varepsilon) \leq \codH{p}_{\mm, \delta'}(A_\tau \cap E_j^\varepsilon) + \varepsilon_j$ for every $j$, where $\varepsilon_j:=(\left(j + \frac{1}{4}\right) \varepsilon)^{-1}\,2^{-j}\varepsilon$. Given $\delta < \min\lbrace r,\eta,\delta',\frac{\tau\sfd(x,y)}{2} \rbrace$ we take set of balls $B_{r_i^j}(z_i^j)$ such that $A_\tau \cap E_j^\varepsilon \subseteq \bigcup_i B_{r_i^j}(z_i^j)$, $r_i^j \leq \delta$ and $\sum_i R_{x,y}^L(z_i^j) \frac{\mm(B_{r_i^j}(z_i^j))}{(r_i^j)^p} \leq \mathcal{H}^g(A_\tau \cap E_j^\varepsilon) + \varepsilon$. Computing as before we get
    \begin{equation*}
        \begin{aligned}
        \int_{A_\tau}R_{x,y}^L\,\d\codH{p}_\mm \leq \sum_j \int_{A_\tau \cap E_j^\varepsilon}R_{x,y}^L\,\d\codH{p}_\mm \leq \sum_j \left(j+\frac{1}{4}\right)\varepsilon\,(\codH{p}_{\mm,\delta'}(A_\tau \cap E_j^\varepsilon) + \varepsilon_j).
        \end{aligned}
    \end{equation*}
    The balls $B_{r_i^j}(z_i^j)$ are admissible for the computation of $\codH{p}_{\mm,\delta'}(A_\tau \cap E_j^\varepsilon)$, so we get
    \begin{equation*}
        \begin{aligned}
        \int_{A_\tau}R_{x,y}^L\,\d\codH{p}_\mm &\leq \sum_j \sum_i\left( R_{x,y}^L(z_i^j) +\frac{3}{2}\varepsilon + \varepsilon'\right)\,\frac{\mm(B_{r_i^j}(z_i^j))}{(r_i^j)^p} + \varepsilon \\
        &\leq \sum_j \left(1 + \left(\frac{3}{2}\varepsilon + \varepsilon'\right)M\right)\mathcal{H}^g(A_\tau \cap E_j^\varepsilon) + \varepsilon,
        \end{aligned}
    \end{equation*}
    where $M$ is such that $M^{-1} \le R_{x,y}^L \le M$ on $\overline{B}_r(A_\tau) \cap \overline{B}_{x,y}^L$. By treating separately odd and even indexes, arguing as in the first part of the proof, we get
    \begin{equation*}
        \begin{aligned}
        \int_{A_\tau}R_{x,y}^L\,\d\codH{p}_\mm \leq 2\left(1 + \left(\frac{3}{2}\varepsilon + \varepsilon'\right)M\right)\mathcal{H}^g(A_\tau) + \varepsilon.
        \end{aligned}
    \end{equation*}
    Now we can take $\varepsilon, \varepsilon', \tau \to 0$ to get the thesis.
\end{proof}

Putting everything together we finally get the equivalent of Corollary \ref{cor:perimeter_comparison_Poincaré} for the codimension Hausdorff measures.

\begin{corollary}
\label{cor:codH_comparison_Poincaré}
Let $(\X,\sfd,\mm)$ be a $q$-{\rm PI} space for some $q \ge 1$. Let $x,y\in \X$, $L,p\geq 1$ and $\Omega \in \SS_{\textup{top}}(x,y)$. Then
    \begin{equation*}
    \frac{1}{C}\,\codH{p}_{\mm_{x,y}^L}(\partial \Omega) \le \int_{\partial \Omega} R_{x,y}^L\,\d\codH{p}_\mm \le C\,\codH{p}_{\mm_{x,y}^L}(\partial \Omega),
    \end{equation*}
    where $C = C(C_D,C_P,\lambda)$.
\end{corollary}
\begin{proof}
    It follows by the usual combination of Remarks \ref{rem:bilipschitz_equivalence} and \ref{rem:biLipschitz_change} together with the lemmas \ref{lemma:comparison_hausdorff_1} and \ref{lemma:comparison_hausdorff_2}.
\end{proof}

\section{Proof of the main theorem}
\label{sec:main_theorem}
We now provide the proof of our main result, that we write here in full generality. 
\begin{theorem}
\label{theo:main-intro-p=1-riproposed}
Let $(\X,\sfd,\mm)$ be a doubling metric measure space. Then the following conditions are quantitatively equivalent:
\begin{itemize}
    \item[(PI)] $1$-Poincar\'{e} inequality;
    \item[(BP)] $\exists c>0, L\geq 1$ such that $\int R_{x,y}^L\,\d {\rm Per}_\mm(\Omega, \cdot) \ge c$ for every $x, y \in \X$ and $\Omega \in \SS_{\textup{top}}(x,y)$;
    \item[(BP$_\textup{R}$)] $\exists c>0, L\geq 1$ such that ${\rm Per}_{\mm_{x,y}^L}(\Omega) \ge c$ for every $x, y \in \X$ and $\Omega \in \SS_{\textup{top}}(x,y)$;
    \item[(BC)] $\exists c>0, L\geq 1$ such that $\textup{cap}_{\mm_{x,y}^L}(\Omega) \ge c$ for every $x, y \in \X$ and   $\Omega \in \SS_{\textup{top}}(x,y)$;
    %\item[(e-BC)] $\exists c>0, L\geq 1$ such that $\textup{cap}_{\mm_{x,y}^L}(\partial^e \Omega) \ge c$ for every $x, y \in \X$ and for every $\Omega \in \SS_{\textup{top}}(x,y)$;
    \item[(BMC)] $\exists c>0, L\geq 1$ such that $(\mm_{x,y}^L)^{+,1}(\Omega) \ge c$ for every $x, y \in \X$ and  $\Omega \in \SS_{\textup{top}}(x,y)$;
    %\item[(e-BMC)] $\exists c>0, L\geq 1$ such that $(\mm_{x,y}^L)^{+,1}(\partial^e \Omega) \ge c$ for every $x, y \in \X$ and for every $\Omega \in \SS_{\textup{top}}(x,y)$;
    \item[(BH)] $\exists c>0, L\geq 1$ such that $\int_{\partial \Omega} R_{x,y}^L\,\d \codH{1}_\mm \ge c$ for every $x, y \in \X$ and  $\Omega \in \SS_{\textup{top}}(x,y)$;
    \item[(BH$^e$)] $\exists c>0, L\geq 1$ such that $\int_{\partial^e \Omega} R_{x,y}^L\,\d \codH{1}_\mm \ge c$ for every $x, y \in \X$ and  $\Omega \in \SS_{\textup{top}}(x,y)$;
    \item[(BH$_\textup{R}$)] $\exists c>0, L\geq 1$ such that ${\codH{1}_{\mm_{x,y}^L}}(\partial \Omega) \ge c$ for every $x, y \in \X$ and   $\Omega \in \SS_{\textup{top}}(x,y)$;
    \item[(BH$^e_\textup{R}$)] $\exists c>0, L\geq 1$ such that ${\codH{1}_{\mm_{x,y}^L}}(\partial^e \Omega) \ge c$ for every $x, y \in \X$ and   $\Omega \in \SS_{\textup{top}}(x,y)$;
    \item[(BAM)] $\exists c>0, L\geq 1$ such that ${\rm AM}_1(\partial \Omega, \mm_{x,y}^L) \ge c$ for every $x, y \in \X$ and  $\Omega \in \SS_{\textup{top}}(x,y)$;
    \item[(BAM$^e$)] $\exists c>0, L\geq 1$ such that ${\rm AM}_1(\partial^e \Omega, \mm_{x,y}^L) \ge c$ for every $x, y \in \X$ and  $\Omega \in \SS_{\textup{top}}(x,y)$;
    \item[(BAM$^\pitchfork$)] $\exists c>0, L\geq 1$ such that ${\rm AM}_1^\pitchfork(\Omega, \mm_{x,y}^L) \ge c$ for every $x, y \in \X$ and $\Omega \in \SS_{\textup{top}}(x,y)$.
\end{itemize}
\end{theorem}

% \begin{definition}[Riesz-PtPI$_p$]
% \label{def:equivalent_riesz_lowerbounds}
% We define some conditions.
% \begin{itemize}
%     \item[i)] For every open set $\Omega$ such that  $x\in \Omega$ and $y \notin \bar{\Omega}$, we have that, defining $\nu:=R_{x,y}\mm$
%     \begin{equation}
%     \label{eq:Minkowski_content_hypersurface}
%         \nu^+(\partial \Omega) \ge C.
%     \end{equation}
%     \item[ii)] For every open set $\Omega$ such that  $x\in \Omega$ and $y \notin \bar{\Omega}$, we have that, defining $\nu:=R_{x,y}\mm$
%     \begin{equation}
%     \label{eq:Minkowski_content_hypersurface_essential}
%         \nu^+(\partial^e \Omega) \ge C.
%     \end{equation}
%     \item[iii)] For every open set $\Omega$ such that  $x\in \Omega$ and $y \notin \bar{\Omega}$, we have that
% \begin{equation}
% \label{eq:codHausd_hypersurface}
%     \int_{\partial \Omega} R_{x,y} \, \d \codH{p} \ge C \sfd(x,y)^{1-p}.
% \end{equation}
%     \item[iv)] For every open set $\Omega$ such that  $x\in \Omega$ and $y \notin \bar{\Omega}$, we have that
% \begin{equation}
% \label{eq:codHausd_hypersurface_essential}
%     \int_{\partial^e \Omega} R_{x,y} \, \d \codH{p} \ge C \sfd(x,y)^{1-p}.
% \end{equation}
%     \item[v)] For every open set $\Omega$ such that  $x\in \Omega$ and $y \notin \bar{\Omega}$, we have that
% \begin{equation}
% \label{eq:codHausd_hypersurface_essential}
%     \int_\X R_{x,y} \, \d {\rm Per}(E,\cdot) \ge C.
% \end{equation}
% \end{itemize}

% \end{definition}

With respect to the acronyms in the introduction, we added the superscript $e$ when we deal with the essential boundary and the subscript R when we consider the metric measure space $(\X,\sfd,\mm_{x,y}^L)$, namely with the weighted measure with Riesz potential $\mm_{x,y}^L$.
The scheme in Figure \ref{fig:strategy_proof} outlines the strategy of the proof.\\

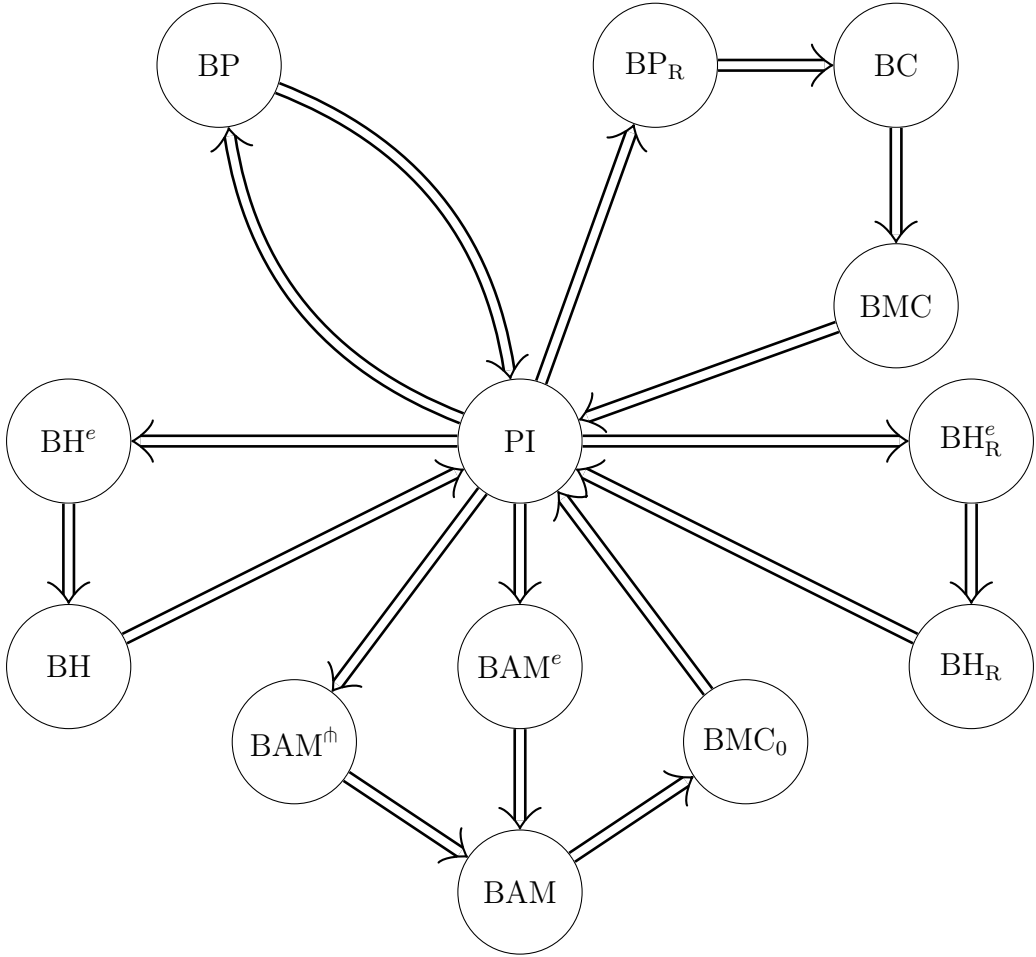
\begin{figure}[h!]
\centering

\tikzset{
    impl/.style = {-{Classical TikZ Rightarrow[length=3mm]}},
}

\begin{tikzpicture}[
node distance = 13mm,
     C/.style = {%C: as circle
                 circle, draw, minimum size=4em, inner sep=2pt},
every edge/.style = {draw, line width=1pt, double distance=3pt},
every edge quotes/.style = {auto, font=\footnotesize, inner sep=1pt, sloped}
                        ]
                        
\node[C] at (0,0) (PI) {$\text{PI}$};
\node[C] at (-4,5) (BP) {$\text{BP}$};
\node[C] at (1.8,5) (BP') {$\text{BP}_\text{R}$};
\node[C] at (5,1.8) (BMC) {$\text{BMC}$};
\node[C] at (5,5) (BC) {$\text{BC}$};
\node[C] at (-6,0) (e-BH) {$\text{BH}^e$};
\node[C] at (-6,-3) (BH) {$\text{BH}$};
\node[C] at (6,0) (e-BH') {$\text{BH}^e_{\text{R}}$};
\node[C] at (6,-3) (BH') {$\text{BH}_{\text{R}}$};
\node[C] at (-3,-4) (BAM') {$\text{BAM}^\pitchfork$};
\node[C] at (0,-3) (e-BAM) {$\text{BAM}^e$};
\node[C] at (3,-4) (BMC0) {$\text{BMC}_0$};
\node[C] at (0,-6) (BAM) {$\text{BAM}$};
\path   (BH) edge[impl]    (PI)
        (BH') edge[impl]    (PI)
        (PI) edge[impl]   (e-BH)
        (PI) edge[impl]   (e-BH')
        (PI) edge[bend left, impl]   (BP)
        (PI) edge[impl]   (BAM')
        (PI) edge[impl]   (e-BAM)
        (PI) edge[impl]   (BP')
        (BP) edge[bend left, impl] (PI)
        (BP') edge[impl]   (BC)
        (BC) edge[impl]   (BMC)
        (BMC) edge[impl]   (PI)
        (BAM') edge[impl]   (BAM)
        (e-BAM) edge[impl]   (BAM)
        (BAM) edge[impl]   (BMC0)
        (BMC0) edge[impl]   (PI)
        (e-BH) edge[impl]  (BH)
        (e-BH') edge[impl]  (BH');
    \end{tikzpicture}
\caption{Logical scheme of the proof of Theorem \ref{theo:main-intro-p=1-riproposed}.}
\label{fig:strategy_proof}
\end{figure}

\noindent The node BMC$_0$ appearing at the bottom right is the following auxiliary property:
\begin{itemize}
    \item[(BMC$_0$)] $\exists c>0, L\geq 1$ such that $(\mm_{x,y}^L)^{+,1}(\partial\Omega) \ge c$ for every $x, y \in \X$ and $\Omega \in \SS_{\textup{top}}(x,y)$ such that $\mm_{x,y}^L(\partial \Omega) = 0$.
\end{itemize}

%\[\begin{tikzcd}
%	& {\text{(BPSS)}} & {\text{(BHSSess)}} & {\text{(BHSS)}} \\
%	{\text{(PI)}} & {\text{(BPSS2)}} & {\text{(BHSS2ess)}} & {\text{(BHSS2)}} & {\text{(BMCSS)}} & {\text{(PI)}} \\
%	& {\text{(BMSS)}} && {\text{(BCSS)}} \\
%	& {\text{(BMSSess)}} && {\text{(BCSSess)}} & {\text{(BMCSSess)}}
%	\arrow[tail reversed, from=2-1, to=2-2]
%	\arrow[tail reversed, from=2-1, to=1-2]
%	\arrow[from=1-2, to=1-3]
%	\arrow[from=1-3, to=1-4]
%	\arrow[from=2-2, to=2-3]
%	\arrow[from=2-3, to=2-4]
%	\arrow[from=2-4, to=2-5]
%	\arrow[from=2-5, to=2-6]
%	\arrow[from=3-2, to=3-4]
%	\arrow[from=3-4, to=2-5]
%	\arrow[from=4-4, to=4-5]
%	\arrow[from=4-5, to=2-5]
%	\arrow[from=4-2, to=4-4]
%%	\arrow[from=2-1, to=4-2]
%	\arrow[from=4-2, to=3-2]
%%	\arrow[from=1-4, to=2-6]
%\end{tikzcd}\]

\begin{proof}
    The first cycle of implications we prove is the bottom one.
    \vspace{2mm}
    
    \noindent {\color{blue}(PI)} $\Rightarrow$ {\color{blue}(BAM$^\pitchfork$)} and {\color{blue}(BAM$^e$)}. Let $x,y\in \X$ and let $\alpha \in \mathcal{P}(\Gamma_{x,y}^L)$ be a $1$-pencil as in Proposition \ref{prop:Poincaré_equivalences_pencil} with constants $C >0, L\geq 1$ depending quantitatively on the structural constants. Let us estimate the approximate $1$-modulus of the family $\Gamma_{x,y}^L$. For every sequence of functions $\rho_j$ such that $\limi_{j\to +\infty} \int_\gamma \rho_j\geq 1$ for every $\gamma \in \Gamma_{x,y}^L$ we have
    \begin{equation*}
        \begin{aligned}
            1 \leq \int \limi_{j\to +\infty}\int_\gamma \rho_j\,\d\alpha \leq \limi_{j\to +\infty} \int \int_\gamma \rho_j \,\d\alpha \leq C \limi_{j\to +\infty} \int_\X \rho_j\,\d\mm_{x,y}^L
        \end{aligned}
    \end{equation*}
    by Fatou's lemma. By the arbitrariness of the sequence $\rho_j$ we have ${\rm AM}_1(\Gamma_{x,y}^L, \mm_{x,y}^L) \geq c$,
    where $c=1/C$. Let us fix now $\Omega \in \SS_{\text{top}}(x,y)$. By definition $\Gamma_{x,y}^L \subseteq \Gamma(\Omega)^\pitchfork$ because $x=\gamma_0 \in I_\Omega$ and $y=\gamma_1\in O_\Omega$ for every $\gamma \in \Gamma_{x,y}^L$. Therefore ${\rm AM}_1^\pitchfork(\Omega, \mm_{x,y}^L) \geq c$. This shows {\color{blue}(BAM$^\pitchfork$)}. \\
    In order to prove {\color{blue}(BAM$^e$)} we fix $x,y\in\X$ and we consider again a $1$-pencil $\alpha \in \mathcal{P}(\Gamma_{x,y}^L)$ with constants $C$ and $L$. Then ${\rm AM}_1(\Gamma_{x,y}^L, \mm_{x,y}^L) \geq c$ as shown before, where $c=1/C$. We choose $L'\geq L$ such that both $B_{2L'\sfd(x,y)}(x)$ and $B_{2L'\sfd(x,y)}(y)$ have finite $\mm$-perimeter and so also ${\rm Per}_\mm(B_{x,y}^{L'}, \X) < \infty$ by Remark \ref{rem:properties_perimeter}.(iii): this is possible by Lemma \ref{lemma:balls_have_finite_perimeter}. We can choose $L'$ arbitrarily close to $L$ and to fix the constants we suppose $L' < L+1$. A straightforward computation shows that ${\rm AM}_1(\Gamma_{x,y}^{L'}, \mm_{x,y}^{L'}) \geq c$ as well. We fix a separating set $\Omega \in \SS_{\text{top}}(x,y)$ and we consider two cases: either ${\rm Per}_\mm(\Omega,B_{x,y}^{L + 1}) < \infty$ or not. If ${\rm Per}_\mm(\Omega,B_{x,y}^{L + 1}) = \infty$ then Proposition \ref{prop:finite_perimeter_equivalences} implies that ${\rm AM}_1(\partial^e\Omega, \mm_{x,y}^{L + 1}) = \infty$, so there is nothing else to prove. Suppose now ${\rm Per}_\mm(\Omega,B_{x,y}^{L + 1}) < \infty$. We claim that ${\rm Per}_\mm(\Omega \cap B_{x,y}^{L'}, \X) < \infty$. For this we argue as follows.     
    By \cite[Proposition 1.16]{BonPasqRajala2020} we have
    $\partial^e(\Omega \cap B_{x,y}^{L'}) \subseteq \partial^e\Omega \cup \partial^e B_{x,y}^{L'}$. Together with the fact that $\partial^e(\Omega \cap B_{x,y}^{L'}) \subseteq \overline{B}_{x,y}^{L'} \subseteq B_{x,y}^{L+1}$ we have  
    \begin{equation}
        \label{eq:partial_boundaries_union}
        \partial^e(\Omega \cap B_{x,y}^{L'}) \subseteq ((\partial^e\Omega)\cap B_{x,y}^{L + 1}) \cup \partial^e B_{x,y}^{L'}.
    \end{equation}
    We remark that in \cite{BonPasqRajala2020} the sets are assumed to have finite perimeter but this assumption is not used in the proof. Therefore 
    $$\codH{1}_\mm(\partial^e(\Omega \cap B_{x,y}^{L'})) \leq \codH{1}_\mm((\partial^e\Omega)\cap B_{x,y}^{L + 1}) + \codH{1}_\mm(\partial^e B_{x,y}^{L'}).$$
    Both terms on the right are finite because of Proposition \ref{prop:finite_perimeter_equivalences} and our assumptions. Hence also the left hand side term is finite, implying that ${\rm Per}_\mm(\Omega\cap B_{x,y}^{L'}, \X) < \infty$, again by Proposition \ref{prop:finite_perimeter_equivalences}.
    So $\Omega \cap B_{x,y}^{L'}$ is a bounded set of finite perimeter in $\X$. By \cite[Corollary 6.4]{LahtiShanmu2017} we have ${\rm Mod}_1(\Gamma(\Omega \cap B_{x,y}^{L'})^\pitchfork \setminus \Gamma(\partial^e(\Omega \cap B_{x,y}^{L'})), \mm) = 0$. Hence ${\rm AM}_1(\Gamma(\Omega \cap B_{x,y}^{L'})^\pitchfork \setminus \Gamma(\partial^e(\Omega \cap B_{x,y}^{L'})), \mm) = 0$. Now we observe that $\Gamma_{x,y}^{L'} \subseteq \Gamma(\Omega \cap B_{x,y}^{L'})^\pitchfork$, because $\gamma_0 = x\in I_{\Omega \cap B_{x,y}^{L'}}$ and $\gamma_1 = y\in O_{\Omega \cap B_{x,y}^{L'}}$ for every $\gamma \in \Gamma_{x,y}^{L'}$. Thus
    \begin{equation}
        \label{eq:AM_pass_boundary}
        {\rm AM}_1(\Gamma_{x,y}^{L'} \setminus \Gamma(\partial^e(\Omega \cap B_{x,y}^{L'})), \mm) = 0
    \end{equation}
    as well. Notice also that every curve $\gamma \in \Gamma_{x,y}^{L'}$ cannot intersect $\partial^e B_{x,y}^{L'} \subseteq \partial B_{x,y}^{L'}$ because $\sfd(x,\gamma_t) \leq \ell(\gamma) \leq L'\sfd(x,y)$ for every $t\in [0,1]$, while every point of $\partial B_{x,y}^{L'}$ is at distance at least $2L'\sfd(x,y)$ from $x$. However, by \eqref{eq:AM_pass_boundary} and \eqref{eq:partial_boundaries_union}, ${\rm AM}_1$-a.e. curve of $\Gamma_{x,y}^{L'}$, computed with respect to $\mm$, has to intersect $((\partial^e\Omega)\cap \overline{B}_{x,y}^{L'}) \cup \partial^e B_{x,y}^{L'}$. This means that ${\rm AM}_1$-a.e. curve of $\Gamma_{x,y}^{L'}$ has to intersect $(\partial^e\Omega) \cap \overline{B}_{x,y}^{L'}$, i.e.
    $${\rm AM}_1(\Gamma_{x,y}^{L'} \setminus \Gamma((\partial^e\Omega) \cap \overline{B}_{x,y}^{L'}), \mm) = 0.$$   
    Since $\min\lbrace \sfd(x,\partial \Omega), \sfd(y,\partial \Omega)\rbrace > 0$ then there exists $M \geq 1$ such that $R_{x,y}^{L'} \leq M$ on a neighbourhood of $(\partial \Omega) \cap \overline{B}_{x,y}^{L'}$. Therefore, by a direct computation and an application of Lemma \ref{lemma:AM_localized_near_boundary}, we also have 
    $${\rm AM}_1(\Gamma_{x,y}^{L'} \setminus \Gamma((\partial^e\Omega) \cap \overline{B}_{x,y}^{L'}), \mm_{x,y}^{L'}) = 0.$$
    Using that ${\rm AM}_1(\cdot, \mm_{x,y}^{L'})$ is an outer measure and is monotone with respect to inclusion of families of paths we get
    $${\rm AM}_1(\Gamma(\partial^e\Omega), \mm_{x,y}^{L+1}) \geq {\rm AM}_1(\Gamma(\partial^e\Omega), \mm_{x,y}^{L'}) \geq {\rm AM}_1(\Gamma_{x,y}^{L'}, \mm_{x,y}^{L'}) \geq c.$$
    This shows {\color{blue}(BAM$^e$)} with constants $L+1$ and $c$.
    \vspace{2mm}
    
    \noindent Each of {\color{blue}(BAM$^\pitchfork$)} and {\color{blue}(BAM$^e$)} $\Rightarrow$ {\color{blue}(BAM)}. This is immediate from the fact that both $\Gamma(\partial^e\Omega)$ and $\Gamma(\Omega)^\pitchfork$ are included in $\Gamma(\partial \Omega)$.
    \vspace{2mm}
    
    \noindent {\color{blue}(BAM)} $\Rightarrow$ {\color{blue}(BMC$_0$)}. This is a direct consequence of Lemma \ref{lemma:approximate_modulus_implies_minkowski}.
    \vspace{2mm}
    
    \noindent {\color{blue}(BMC$_0$)} $\Rightarrow$ {\color{blue}(PI)}. Let $u\colon \X \to \R$ be a bounded Lipschitz function such that $u\geq 0$ and let $x,y\in \X$. We want to prove the pointwise estimate \eqref{eq:Riesz_PtPI}. If $u(x) = u(y)$ there is nothing to prove. Otherwise we can suppose, without loss of generality, that $u(x) < u(y)$.  
    We first apply the coarea inequality for the codimension $1$-Hausdorff measure relative to the measure $\mm_{x,y}^L$ (Proposition \ref{prop:coarea_inequality_hausdorffmeasure}) to get
    $$\int_0^{+\infty} \codH{1}_{\mm_{x,y}^L}(\lbrace u = t\rbrace)\,\d t \leq 2\int_\X \lip u\, \d \mm_{x,y}^L.$$
    Since $\int_\X \lip u \,\d \mm^L_{x,y} \le {\rm Lip}(u) \mm^L_{x,y}(\X) < \infty$, the right hand side is finite, then $\codH{1}_{\mm_{x,y}^L}(\lbrace u = t\rbrace)$ is finite as well for almost every $t$.
    Then $\mm_{x,y}^L(\{u=t\})=0$ for almost every $t$, by Proposition \ref{prop:comparison_finiteness_Hcod_mm}.   
    We focus on the sets $\Omega_t := \lbrace u \geq t\rbrace$ for $u(x) < t < u(y)$. Clearly $\Omega_t \in \SS_{\text{top}}(x,y)$ for every such $t$. Moreover $\partial \Omega_t \subseteq \lbrace u = t \rbrace$, so it has $\mm_{x,y}^L$-measure zero for almost every $t$. In particular for almost every $t\in (u(x),u(y))$ we can apply the assumption {\color{blue} (BMC$_0$)} to get $(\mm_{x,y}^L)^{+,1}(\partial \Omega_t) \geq c$. We are now in position to apply the coarea inequality for the Minkowski content as stated in \eqref{eq:coareain3_KL}  to get
    $$c\,\vert u(x) - u(y) \vert \leq \int_{u(x)}^{u(y)} (\mm_{x,y}^L)^{+,1}(\partial \Omega_t)\,\d t \leq 2\int_\X \lip u\,\d\mm_{x,y}^L.$$
    The pointwise estimate follows with $C = 2/c$, for Lipschitz, bounded, nonnegative functions. By translation invariance we immediately have the same inequality for all bounded, Lipschitz functions. Let ${u \in \rm Lip}(\X)$. Let $u_n = \max\{\min\{u,n\},-n\}$. Notice that $u_n \in {\rm Lip}(\X)$ and it is bounded. We have that $|u_n(x)-u_n(y)|$ converges to $|u(x)-u(y)|$ and $\lip u_n \le \lip u$. Thus, applying the estimate to $u_n$ and taking the limit we conclude.
    \vspace{2mm}

    \noindent We now move to the right cycle of implications.
    \vspace{2mm}

    \noindent {\color{blue}(PI)} $\Rightarrow$ {\color{blue} (BH$^e_\text{R}$)} $\Rightarrow$ {\color{blue} (BH$_\text{R}$)} $\Rightarrow$ {\color{blue} (PI)}. We have already seen that {\color{blue}(PI)} implies {\color{blue} (BAM$^e$)}. Therefore Corollary \ref{cor:AM<Hcod_Poincaré} gives directly {\color{blue} (BH$^e_\text{R}$)}. It is also obvious that {\color{blue} (BH$^e_\text{R}$)} implies {\color{blue} (BH$_\text{R}$)}. We now suppose {\color{blue} (BH$_\text{R}$)} holds, we take a Lipschitz function $u\colon \X \to \R$, $u\geq 0$ and two points $x,y\in \X$. We want to prove the pointwise estimate \eqref{eq:Riesz_PtPI}. If $u(x) = u(y)$ there is nothing to prove, otherwise once again we can suppose without loss of generality that $u(x) < u(y)$. The sets $\Omega_t := \lbrace u \geq t \rbrace$ belong to $\SS_\text{top}(x,y)$ for all $t\in (u(x),u(y))$. Moreover $\partial \Omega_t \subseteq \lbrace u = t \rbrace$. So we can apply the coarea inequality \eqref{eq:coareain2} of Proposition \ref{prop:coarea_inequality_hausdorffmeasure} with respect to the measure $\mm_{x,y}^L$ to get
    $$c\,\vert u(x) - u(y) \vert \leq \int_{u(x)}^{u(y)} \codH{1}_{\mm_{x,y}^L}(\lbrace u = t \rbrace) \,\d t \leq 2\int_\X \lip u \,\d\mm_{x,y}^L.$$
    Once again we can use this coarea inequality because of Proposition \ref{prop:properties_mxy}.
    Therefore the pointwise estimate follows with $C = 2/c$ for Lipschitz, nonnegative functions. Arguing as before we get the same estimate for all Lipschitz functions.
    \vspace{2mm}

    \noindent The next step is the left cycle of implications.
    \vspace{2mm}

    \noindent {\color{blue}(PI)} $\Rightarrow$ {\color{blue} (BH$^e$)} $\Rightarrow$ {\color{blue} (BH)} $\Rightarrow$ {\color{blue} (PI)}. We already know that {\color{blue}(PI)} implies {\color{blue} (BH$^e_\text{R}$)}. We immediately deduce {\color{blue} (BH$^e$)} by Corollary \ref{cor:codH_comparison_Poincaré}. This gives {\color{blue} (BH)}. It remains to prove {\color{blue}(BH)} $\Rightarrow$ {\color{blue}(PI)}.
    We take a Lipschitz function $u\colon \X \to \R$, $u\geq 0$ and two points $x,y\in \X$. We want to prove the pointwise estimate \eqref{eq:Riesz_PtPI}. We apply the coarea inequality \eqref{eq:coareain2} of Proposition \ref{prop:coarea_inequality_hausdorffmeasure} with respect to the measure $\mm$ and Borel function $R_{x,y}^L$ to get
    $$c\,\vert u(x) - u(y) \vert \leq \int_{u(x)}^{u(y)} \int_{\lbrace u = t \rbrace} R_{x,y}^L\,\d \codH{1}_{\mm} \,\d t \leq 2\int_\X R_{x,y}^L\,\lip u \,\d\mm.$$
    \vspace{2mm}

    \noindent We now go to {\color{blue}(PI)} $\Leftrightarrow$ {\color{blue}(BP)}.
    Suppose {\color{blue} (PI)} holds. We already know that {\color{blue} (BH$^e$)} holds. We fix $x,y\in\X$ and $\Omega \in \SS_{\text{top}}(x,y)$. If ${\rm Per}_\mm(\Omega,B_{x,y}^L) = \infty$ there is nothing to prove. 
    So we can restrict to the case ${\rm Per}(\Omega,B_{x,y}^L)< \infty$. By Proposition \ref{prop:representation_perimeter_ambrosio} we know that ${\rm Per}_\mm(\Omega, \cdot) = \theta \codH{1}_\mm\restr{\partial^e\Omega}(\cdot)$ as measures on $B_{x,y}^L$, with $c_1\leq \theta \leq C_D$ and $c_1$ depends only on the structural constants. Therefore
    \begin{equation*}
        \begin{aligned}
            \int_\X R_{x,y}^L \,\d{\rm Per}_\mm(\Omega, \cdot) = \int_{B_{x,y}^L} R_{x,y}^L \,\d{\rm Per}_\mm(\Omega, \cdot) &\geq c_1\int_{\partial^e\Omega \cap B_{x,y}^L} R_{x,y}^L \,\d\codH{1}_\mm \\
            &= c_1\int_{\partial^e\Omega} R_{x,y}^L \,\d\codH{1}_\mm \geq c_1c.
        \end{aligned}
    \end{equation*}
    This shows {\color{blue} (BP)} with constants $c_1c$ and $L$.\\ 
    Suppose now {\color{blue}(BP)}  holds. We want to prove the pointwise estimate \eqref{eq:Riesz_PtPI}.
    Let $x,y \in \X$ and $u \in {\rm Lip}(\X)$. Assume without loss of generality that $u(x) < u(y)$. Consider $u(x)< t < u(y)$ and $\Omega_t: =\{ u \le t\} \in {\rm SS}_{{\rm top}}(x,y)$. Then
    \begin{equation*}
    c \le \int R_{x,y}^L\, \d {\rm Per}_\mm(\{ u \le t \},\cdot)=\int R_{x,y}^L\, \d {\rm Per}_\mm(\{ u > t \},\cdot).
    \end{equation*}
    where the last equality follows by Remark \ref{rem:properties_perimeter}.(iii).
    Hence, using \eqref{eq:coarea_inequality_BV},
    \begin{equation*}
    c\,|u(x)-u(y)| \le \int_{u(x)}^{u(y)} \int R_{x,y}^L\, \d {\rm Per}_\mm(\{ u > t \},\cdot) \,\d t \le 2 \int R_{x,y}^L\,\lip u\,\d \mm.
    \end{equation*}
    \vspace{2mm}

    \noindent The remaining cycle of implications is {\color{blue}(PI)} $\Rightarrow$ {\color{blue} (BP$_\text{R}$)} $\Rightarrow$ {\color{blue} (BC)} $\Rightarrow$ {\color{blue} (BMC)} $\Rightarrow$ {\color{blue} (PI)}. \\
    Suppose {\color{blue}(PI)} holds. We already know that {\color{blue}(BP)} holds with constants $c>0$, $L\geq 1$.
    Let $x,y \in \X$ and $\Omega \in {\rm SS}_{\rm top}(x,y)$. We apply Corollary \ref{cor:perimeter_comparison_Poincaré} to get
    \begin{equation*}
        c \le \int R_{x,y}^L \d {\rm Per}_\mm(\Omega, \cdot) \le C'\,{\rm Per}_{\mm_{x,y}^L}(\Omega, B_{x,y}^L)\le C'\,{\rm Per}_{\mm_{x,y}^L}(\Omega, \X).
    \end{equation*}
    Therefore {\color{blue} (BP$_\text{R}$)} holds with constants $c/C'$ and $L$.\\    
    The implications {\color{blue} (BP$_\text{R}$)} $\Rightarrow$ {\color{blue} (BC)} $\Rightarrow$ {\color{blue} (BMC)} follow directly from Lemma \ref{lemma:perimeter<capacity<minkowski}, because the separating sets are closed and their intersection with $\overline{B}_{x,y}^L$ is compact. To be precise, in the application of Lemma \ref{lemma:perimeter<capacity<minkowski} we use the fact that the perimeter, the capacity and the Minkowski content of a set $A$ computed in the metric measure space $(\X,\sfd,\mm_{x,y}^L)$ are equal to the perimeter, the capacity and the Minkowski content of the set $A \cap {\rm supp}\,\mm_{x,y}^L=A \cap \overline{B}_{x,y}^L$ computed in the same metric measure space.\\ 
    Now assume {\color{blue} (BMC)} holds and we prove {\color{blue} (PI)}. Let $u\colon \X \to \R$, $u\geq 0$ be a bounded Lipschitz function and let $x,y\in \X$. We want to prove the pointwise estimate \eqref{eq:Riesz_PtPI}. We can assume as usual that $u(x) < u(y)$. The sets $\Omega_t := \lbrace u \geq t \rbrace$ belong to $\SS_\text{top}(x,y)$ for all $t\in (u(x),u(y))$. So we can apply the coarea inequality \eqref{eq:coareain3_AGD} of Proposition \ref{prop:coarea_inequality_minkowski} with respect to the measure $\mm_{x,y}^L$ to get
    $$c\, \vert u(x) - u(y) \vert \leq \int_{u(x)}^{u(y)} (\mm_{x,y}^L)^{+,1}(\lbrace u \geq t \rbrace) \,\d t \leq \int_\X \lip u \,\d\mm_{x,y}^L.$$
    Therefore the pointwise estimate follows with $C = 1/c$ for Lipschitz, nonnegative, bounded functions. As before we get the same estimate for all Lipschitz functions.
\end{proof}
\begin{remark}[Comparison with Keith's characterizations]
\label{rem:keith_similarity}
In the proof of the implication (PI) implies (BAM$^\pitchfork$), we prove that (PI) implies the following condition:
there exists $c>0$ and $L \ge 1$ such that for every $x,y \in \X$ we have $${\rm AM}_{1}(\Gamma_{x,y}^L,\mm_{x,y}^L) \ge c.$$
This condition in the proof implies (BAM$^\pitchfork$) which implies (PI).
Thus the above condition is equivalent to (PI) and share a close similarity with Keith's characterizations in \cite{Kei03}. Indeed, Keith proved that, given $p \ge 1$, a doubling metric measure space satisfies a $p$-Poincar\'{e} inequality if and only if 
$$ {\rm Mod}_p(\Gamma_{x,y}^L,\mm_{x,y}^L) \ge c.$$
Looking at this condition for $p=1$, we have that our condition is apriori weaker than Keith's one, but it turns out to be equivalent.
\end{remark}
We remark that the same scheme of proof shows the following property, holding for every space satisfying a $p$-Poincaré inequality.
\begin{proposition}
\label{prop:PI_implies_pBH}
    Let $(\X,\sfd,\mm)$ be a $p$-\textup{PI} space, $p\geq 1$. Then there exist $c > 0$, $L\geq 1$, depending only on $(C_D,C_P,\lambda)$, such that for every $x,y\in\X$ and every $\Omega \in \SS_{\textup{top}}(x,y)$ it holds
    $$\int_{\partial \Omega} R_{x,y}^L \,\d\codH{p}_\mm \geq c.$$
    In particular $\codH{p}_\mm(\partial \Omega) > 0$.
\end{proposition}
\begin{proof}
    The same proof of {\color{blue} (PI)} implies {\color{blue} (BAM$^\pitchfork$)}, where we use a $p$-pencil instead of a $1$-pencil, shows that we can find $c>0$, $L\geq 1$ depending only on the data such that 
    $${\rm AM}_p^\pitchfork(\Omega, \mm_{x,y}^L) \geq c$$
    for every $\Omega \in \SS_{\text{top}}(x,y)$. Therefore we also have 
    $${\rm AM}_p(\partial\Omega, \mm_{x,y}^L) \geq c.$$
    Corollary \ref{cor:AM<Hcod_Poincaré}, together with Corollary \ref{cor:codH_comparison_Poincaré} gives the thesis. The last part of the statement is obvious.
\end{proof}

The converse is false in full generality for $p>1$ as the next example shows.
\begin{example}
\label{example:sierpinski_p>1}
    Let $(\X,\sfd,\mm)$ be the standard Sierpinski carpet. It is an $s$-Ahlfors regular metric measure space for $s=\frac{\log 8}{\log 3}$. It does not satisfy a $p$-Poincaré inequality for every $p\geq 1$. Since the space is connected then every separating set has no empty boundary. Therefore for every $s' > s$ we have $\codH{s'}(\partial \Omega) \geq  C_{\text{AR}}^{-1}\,\mathcal{H}^{s-s'}(\partial \Omega) = +\infty$, where we used Remark \ref{remark:hcod-p_properties} and the fact that $\partial \Omega$ is not empty. In particular $\int_{\partial \Omega} R_{x,y}^L\,\d\codH{s'}_\mm = +\infty$ for every $\Omega \in \SS_{\textup{top}}(x,y)$ and for every $x,y\in \X$.
\end{example}

Let us comment with some examples for which one of the equivalent properties that are considered in Theorem \ref{theo:main-intro-p=1} can be used to give necessary and sometimes sufficient conditions for the validity of a $1$-Poincaré inequality.

\begin{example}[$1$-Ahlfors regular spaces]
    Let $(\X,\sfd,\mm)$ be a $1$-Ahlfors regular space. By Remark \ref{remark:hcod-p_properties} there exists a constant $C$ depending only on the Ahlfors regularity constant $C_{\textup{AR}}$ such that $C^{-1}\codH{1} \leq \mathcal{H}^0 \leq C \codH{1}$. The following properties are quantitatively equivalent:
    \begin{itemize}
        \item[(i)] $\exists L \geq 1$ such that $\X$ is $L$-quasiconvex;
        \item[(ii)] $\exists L \geq 1$ such that $(\X,\sfd)$ is $L$-biLipschitz equivalent to a geodesic space;
        \item[(iii)] $(\X,\sfd,\mm)$ satisfies a $1$-Poincaré inequality.
    \end{itemize}
    
    \noindent We only need to verify that (i) implies (iii). To do so we verify the (BH) property. Let us consider $x,y\in\X$ and $\Omega \in {\rm SS}_{\rm top}(x,y)$.
    Let $\gamma$ be a $L$-quasigeodesic connecting $x$ to $y$, so that $\gamma \subseteq B_{x,y}^{L}$. Moreover $\gamma \cap \partial \Omega \neq \emptyset$, thus $\gamma \cap (\partial \Omega \cap B_{x,y}^{L}) \neq \emptyset$. By Ahlfors-regularity we have $C_{\textup{AR}}^{-1} \chi_{B_{x,y}^{L}} \le  R^L_{x,y}\chi_{B_{x,y}^{L}} \le C_{\textup{AR}} \chi_{B_{x,y}^{L}}$.
    Thus,
    \begin{equation*}
        \int_{\partial \Omega} R^{L}_{x,y} \d \codH{1} \ge C_{\textup{AR}}^{-1}\, \mathcal{H}^0(B_{x,y}^{L} \cap \partial \Omega) \ge C_{\textup{AR}}^{-1}.
    \end{equation*}
    Examples of $1$-Ahlfors regular spaces include graphs with an upper bound on the valency of the vertices and a lower bound on the length of the edges. Similar kind of simplicial complexes have been studied for instance in \cite{CavSamb22}. A general simplicial complex, even with bounded geometry, does not satisfy in general a $1$-Poincaré inequality.    
    The explicit computation of the best possible Poincaré inequality for simplicial complexes of bounded geometry can be done. One of the motivations for our study was the computation of the best possible Poincaré inequality for geodesically complete metric spaces with curvature bounded above. At the moment this problem is still unsolved.
\end{example}

\begin{example}[Necessary condition in the Ahlfors regular case]
Let $(\X,\sfd,\mm)$ be a $1$-PI space, thus (BH) holds. Then let $x\in \X$ and $\Omega=\overline{B}(x,r)$. For every $y$ such that $\sfd(x,y)>r$ we have that $\Omega \in {\rm SS}_{\rm top}(x,y)$.
For such a given $y$, $\partial \Omega \subseteq B_{x,y}^L$ and (BH) implies $\int_{\partial \Omega} R_{x,y}\,\d \codH{1} \ge C$. Thus
\begin{equation*}
\begin{aligned}
    C & \le \int_{\partial B_r(x)} R_{x,y}\,\d \codH{1} = \int_{\partial B_r(x)} \frac{\sfd(x,z)}{\mm(B_{\sfd(x,z)}(x))}\,\d \codH{1} + \int_{\partial B_r(x)} \frac{\sfd(y,z)}{\mm(B_{\sfd(y,z)}(y))}\,\d \codH{1}\\
    &\le \frac{r}{\mm(B_r(x))}\,\codH{1}(\partial B_r(x)) + \int_{\partial B_r(x)} \frac{\sfd(x,y)+r}{\mm(B_{\sfd(x,y)-r}(y))}\,\d \codH{1}\\
\end{aligned}
\end{equation*}
We further assume now that $\X$ is $s$-Alhfors regular for $s \ge 1$ with constant $C_{\text{AR}} >0$. Then there exists $C_1=C_1(C,C_{\text{AR}})>0$ such that
\begin{equation}
\label{eq:estimate_size_balls}
    C_1 \le r^{1-s}\,\mathcal{H}^{s-1}(\partial B_r(x)) + \frac{\sfd(x,y)+r}{(\sfd(x,y)-r)^{s}}\,\mathcal{H}^{s-1}(\partial B_r(x)).
\end{equation}
Suppose moreover that $(\X,\sfd)$ is not compact and choose $y_n$ such that $\sfd(x,y_n) \to \infty$. Then we get 
\begin{equation*}
    \mathcal{H}^{s-1}(\partial B_r(x)) \ge C_1 \,r^{s-1}\qquad \text{for every }r > 0 \text{ and every } x\in \X.
\end{equation*}
This is a necessary condition for the validity of the $1$-Poincaré inequality of a Ahlfors regular, non compact metric measure space.\\
More generally, if $(\X,\sfd,\mm)$ is Ahlfors regular and $1$-PI, but not necessarely non-compact, we get an infinitesimal condition. Indeed, by taking the limit as $r \to 0$ in \eqref{eq:estimate_size_balls}, we have that for every $x \in \X$
\begin{equation*}
    \limi_{r \to 0} \frac{\mathcal{H}^{s-1}(\partial B_r(x))}{r^{s-1}} \ge C_2,
\end{equation*}
with $C_2$ being equal to $C_1$ if $s>1$ and $C_2=C_1/2$ if $s=1$. The case $s<1$ is excluded because $\X$ has to be quasigeodesic.
\end{example}
\bibliographystyle{alpha}
\bibliography{biblio.bib}
\end{document}